\documentclass[12pt]{amsart}
\usepackage{enumerate,enumitem,amsmath,amssymb,mathrsfs,stmaryrd}
\usepackage{latexsym,epsf,graphicx,comment,appendix}

\newfont{\bb}{msbm10 at 12pt}
\newfont{\tbb}{msbm10 at 8pt}

\def\r{\hbox{\bb R}}

\usepackage[latin1]{inputenc}
\usepackage {amsmath}
\usepackage {amsthm}
\usepackage{times}
\usepackage{amscd}
\usepackage{epsf}

\numberwithin{equation} {section}

\begin{document}
\mbox{}\vspace{0.2cm}\mbox{}

\theoremstyle{plain}\newtheorem{lemma}{Lemma}[section]
\theoremstyle{plain}\newtheorem{proposition}{Proposition}[section]
\theoremstyle{plain}\newtheorem{theorem}{Theorem}[section]

\theoremstyle{plain}\newtheorem*{theorem-m}{Main Theorem}
\theoremstyle{plain}\newtheorem*{theorem*}{Theorem}
\theoremstyle{plain}\newtheorem*{main theorem}{Main Theorem} 
\theoremstyle{plain}\newtheorem*{lemma*}{Lemma}
\theoremstyle{plain}\newtheorem*{claim}{Claim}

\theoremstyle{plain}\newtheorem{example}{Example}[section]
\theoremstyle{plain}\newtheorem{remark}{Remark}[section]
\theoremstyle{plain}\newtheorem{corollary}{Corollary}[section]
\theoremstyle{plain}\newtheorem*{corollary-A}{Corollary}
\theoremstyle{plain}\newtheorem{definition}{Definition}[section]
\theoremstyle{plain}\newtheorem{acknowledge}{Acknowledgment}
\theoremstyle{plain}\newtheorem{conjecture}{Conjecture}

\begin{center}
\rule{15cm}{1.5pt} \vspace{.4cm}

{\bf\Large On Huber's type theorems in general dimensions} 
\vskip .3cm

Shiguang Ma$\, ^\dag$\footnote{The author is the corresponding author and
partially supported by NSFC 11571185} 
and \hskip 0.15cm Jie Qing$\,^\ddag$\footnote{The author is partially supported by NSF DMS-1608782}\\

\vspace{0.3cm} 
\rule{15cm}{1.5pt}
\end{center}

\vskip 0.3cm
\noindent$\mbox{}^\dag$ School of Mathematical Science and LPMC, Nankai University, Tianjin, China; \\e-mail: 
msgdyx8741@nankai.edu.cn 
\vspace{0.2cm}

\noindent $\mbox{}^\ddag$ Department of Mathematics, University of California, Santa Cruz, CA 95064; \\
e-mail: qing@ucsc.edu

\title{}
\begin{abstract} In this paper we present some extensions of the celebrated finite point conformal compactification theorem of 
Huber \cite{Hu57} for complete open surfaces to general dimensions based on the n-Laplace equations in conformal geometry. 
We are able to conclude a domain in the round sphere has to be the sphere deleted finitely many points if it can be endowed with 
a complete conformal metric with the negative part of the smallest Ricci curvature satisfying some integrable conditions. Our 
proof is based on the strengthened version of the Arsove-Huber's type theorem on n-superharmonic functions in our earlier 
work \cite{MQ18}.  Moreover, using p-parabolicity, we push the injectivity theorem of Schoen-Yau to allow some negative curvature 
and therefore establish the finite point conformal compactification theorem for manifolds that have a conformal immersion into the 
round sphere. As a side product we establish the injectivity of conformal immersions from n-parabolicity alone, which is interesting 
by itself in conformal geometry.  
\end{abstract}

\maketitle


\section{Introduction}\label{Sec:Intro}

In this paper we want to present extensions of the celebrated Huber Theorem on complete open surfaces. Huber \cite{Hu57} in 1957 showed that a complete
open surface whose negative part of the Gaussian curvature is integrable is a closed surface with finitely many points removed. This Huber theorem made the highlight of applications of subharmonic functions in differential geometry after the seminal work of Cohn-Vossen \cite{CV35}. In 2 dimensions the Gaussian curvature equation and the Gauss-Bonnet formula are the essentials.
\\

It has been highly desirable to extend this celebrated Huber Theorem to higher dimensions. Notably there is a Huber's type theorem of Chang, 
Qing and Yang \cite{CQY00} in 4 dimensions, based on the so-called $Q$-curvature equations and the involvement of 
$Q$ curvature in Chern-Gauss-Bonnet formula. Also notably, Carron and Herzlich \cite{CH02} (also see \cite{Car20}) obtained a Huber's type theorem 
in general dimensions for domains in a given compact Riemannian manifold that admit complete conformal metrics with growth conditions on the volume 
and integral conditions on Ricci or scalar curvature. 
 \\
 
 It was observed in \cite{BMQ17, MQ18} the following n-Laplace equation in conformal geometry
 \begin{equation}\label{Equ:n-laplace-intro}
 -\Delta_n u + Ric (\nabla u)|\nabla u|^{n-2} = (Ric (\nabla u)|\nabla u|^{n-2})_{g} e^{nu}.
 \end{equation}
 plays a similar role in general dimensions to the Gaussian curvature equation in 2 dimensions, where $Ric(v)$ is the Ricci curvature in the $v$ direction,  
 the term $(|\nabla u|^{n-2}Ric(\nabla u))_g$ is calculated for the conformal metric $g = e^{2u}\bar g$, 
 and other terms are calculated under the background metric $\bar g$. 
Particularly, in \cite{MQ18}, Arsove-Huber's type theorem
 \cite{AH73} for n-superharmonic functions have been established, where Wolff potential theory provides the important analytic tool in general 
 dimensions as the Newton potential did in 2 dimensions.  
 \\
 
In this paper, first, by capacity estimates, we obtain

\begin{theorem}\label{Thm:intro-1}
For $n\geq 3$, let $\Omega$ be a domain in a compact manifold $(M^n, \bar g)$ and endowed with a complete conformal metric 
$g = e^{2u}\bar g$. Assume that,  
\begin{equation*}
R^-_g \in L^p(\Omega, g) \text{ for some $p\in(\frac{n}{2},+\infty]$ and } \left\{
\aligned
 & \text{either } Ric^{-}_g \in L^{\frac{n}{2}}(\Omega, g),\\
 & \text{or } Ric^{-}_g |\nabla u|^{n-2}e^{2u} \in L^{1}(\Omega, \bar g).
\endaligned\right.
\end{equation*}
where $R^-_g$ is the negative part of the scalar curvature $R$. Then $cap_{n}(\partial\Omega,D)=0$, where $D$ is a neighborhood
of $\partial\Omega$ in $M^n$. Hence $(\Omega, g)$ is $n$-parabolic and $\partial\Omega$
is of Hausdorff dimension 0 in $(M^n, \bar g)$. 
\end{theorem}
 
This theorem may be compared with \cite[Corollary 2.2]{CH02}. Notice that \cite[Theorem 1]{Gal88} (cf. Lemma \ref{Lem:gallot}
in Section \ref{Sec:preli}) is not applicable under the assumptions of Theorem \ref{Thm:intro-1}. 
An example, similar to \cite[Theorem 2.7]{Car20}, is discussed at the end of Section \ref{Subsec:domains} to illustrate it is not 
possible to conclude $\partial\Omega$ to be finite under the assumption $Ric^-\in L^\frac n2(\Omega, g)$.
\\

To achieve finite point conformal compactification, we strengthen \cite[Theorem 1.1]{MQ18} in two aspects. The first is to drop the assumption 
of the origin being isolated in \cite[Theorem 1.1]{MQ18} for $m\geq 1$ from the completeness (cf. Theorem \ref{Thm:improvement of theorem 1.1}). 
The second is
to confirm that $m$ is related to the point mass at the origin (cf. Lemma \ref{Lem:m=point mass}). In order to use 
Theorem \ref{Thm:improvement of theorem 1.1} and Lemma \ref{Lem:m=point mass}, 
we introduce the n-superharmonic function $\mathcal{A}(v)$ (cf. \eqref{Equ:definition of v^+}) 
to overshadow the conformal factor $v$ based on the idea in the proof of 
\cite[Theorem 2.4]{KM92} on solving n-Laplace equations with nonnegative Radon measure on the right. For this purpose, 
we first need to show that the conformal factor $v$ satisfies the n-Laplace equation with a signed Radon measure on the right. 
We therefore obtain the following finite point conformal compactification theorem of Huber's type. 

\begin{theorem}\label{Thm:intro-2}
For $n\geq 3$, let $\Omega$ be a domain in the standard unit round sphere $(\mathbb{S}^n, g_{\mathbb{S}})$.
Suppose that, on $\Omega$, there is a conformal metric $g= e^{2u}g_{\mathbb{S}}$ satisfying
$$
\lim_{x\to\partial\Omega}u(x)=+\infty \text{ and } 
Ric^{-}_g|\nabla u|^{n-2}e^{2u} \in L^{1}(\Omega,g_{\mathbb{S}}).
$$
Then $\partial\Omega = \mathbb{S}^n\setminus\Omega$ is a finite point set. 
\end{theorem}

It is easily seen that Theorem \ref{Thm:intro-2} holds when the Ricci is nonnegative outside a compact subset or Cheng-Yau's gradient estimates (cf. \cite[Theorem 2.12, Chapter VI]{SY94} and \cite[Lemma 2.3]{CQY00}) applies. Here we use \cite[Proposition 8.1]{CHY04} to derive 
$\lim_{x\to\partial\Omega}u(x)=+\infty$. We therefore have the following finite point conformal compactification theorem of Huber's type in a more intrinsically 
geometric fashion. 
\begin{corollary}\label{Cor:intro-1}
For $n\geq 3$, let $\Omega$ be a domain in the standard unit round sphere $(\mathbb{S}^n, g_{\mathbb{S}})$.
Suppose that, on $\Omega$, there is a complete conformal metric $g= e^{2u}g_{\mathbb{S}}$ satisfying either $Ric_g$ is nonnegative outside
a compact subset or 
\begin{enumerate}
\item $Ric^{-}_g \in L^{1}(\Omega, g) \cap L^\infty(\Omega, g)$
\item $R_g \in L^{\infty}(\Omega, g) \text{ and } |\nabla^g R_g|\in L^{\infty}(\Omega, g)$. 
\end{enumerate}
Then $\partial\Omega = \mathbb{S}^n\setminus\Omega$ is a finite point set. 
\end{corollary}

Corollary \ref{Cor:intro-1} should compared with \cite[Theorem 2]{CQY00}. 
\cite[Theorem 2]{CQY00} uses the $Q$-curvature equation in 4 dimensions 
while Corollary \ref{Cor:intro-1} uses the n-Laplace equation \eqref{Equ:n-laplace-intro} in general dimenions.  At the end of Section \ref{Subsec:finite point}, 
we will present a detailed comparison of Theorem \ref{Thm:intro-2} and Corollary \ref{Cor:intro-1} with \cite[Theorem 2.1]{CH02} 
and \cite[Corollary 2.4]{Car20}.
\\

To push for locally conformally flat manifolds that are not a priorly known to be domains in $\mathbb{S}^n$, 
we develop injectivity theorems for conformal immersions 
following the PDE approach in the seminal paper \cite{SY88}. To substitute the nonnegativity of the scalar curvature, we consider volume growth
conditions, p-parabolicity, and integral lower bounds of Ricci curvature. First we establish, as a side product, 
a very clean injectivity theorem based on
n-parabolicity alone, which should be interesting by itself in conformal geometry and geometric group theory.

\begin{theorem}\label{Thm:intro-3}
 Suppose that $(M^n, g)$ ($n\geq 3$) is a complete noncompact manifold that has a conformal immersion
$$
\Phi: (M^n, g)\to(\mathbb{S}^n, g_{\mathbb{S}}).
$$
Assume that $(M^n, g)$ is n-parabolic. Then $\Phi$ is injective.
\end{theorem}

Following closely from \cite[Theorem 3.1]{SY88}, we prove an injectivity theorem using integral lower bound on Ricci curvature, where the volume growth estimates \cite{Gal88} are employed. Namely,

\begin{theorem}\label{Thm:intro-4}
Suppose that $(M^n, g)$ is a complete, noncompact manifold. And suppose that there is a conformal immersion
$$
\Phi: (M^n, g) \to (\mathbb{S}^n, g_{\mathbb{S}}).
$$
Assume that $Ric_g^-\in L^\frac p2(M, g)$ for some $p\in (n, \infty]$. Then $\Phi$ is injective provided that
$$
d(M)< \frac {(n-2)^2p}{(p-2)n}.
$$
\end{theorem}
For the definition of $d(M)$, please see \eqref{Equ:d(m)} and \cite{SY88}. Again, with a refined argument of the proof 
of \cite[Proposition 2.4]{SY88}, we have

\begin{proposition}\label{Prop:intro-1}
Suppose that $(M^n, g)$ ($n\geq 3$) is a manifold that has a conformal immersion 
$$
\Phi: (M^n, g) \to (\mathbb{S}^n, g_{\mathbb{S}}).
$$ 
Assume that $R_g^-\in L^\frac n2(M, g)$. Then $d(M) \leq \frac n2$.
\end{proposition}

Combining Theorem \ref{Thm:intro-4} and Proposition \ref{Prop:intro-1} we establish

\begin{corollary}\label{Cor:intro-2}
Suppose that $(M^n, g)$ is a complete, noncompact manifold. And suppose that there is a conformal immersion
$$
\Phi: (M^n, g) \to (\mathbb{S}^n, g_{\mathbb{S}}).
$$
Assume that $Ric^-_g \in L^\frac p2(M, g)\cap L^\frac n2(M, g)$ for some $p\in (n, \infty]$. Then $\Phi$ is injective provided that
$n\geq 5$.
\end{corollary}

Finally, in the light of Corollary \ref{Cor:intro-1} and Corollary \ref{Cor:intro-2}, we are able to show

\begin{theorem-m}\label{Thm:intro-main}
Suppose that $(M^n, g)$ ($n\geq 5$) is a complete noncompact manifold. And suppose 
that there is a conformal immersion
$$
\Phi: (M^n, g) \to (\mathbb{S}^n, g_{\mathbb{S}}).
$$ 
Assume that either $Ric_g$ is nonnegative outside a compact subset or 
\begin{enumerate}
\item $Ric^{-}_g \in L^{1}(\Omega, g) \cap L^\infty(\Omega, g)$
\item $R_g \in L^{\infty}(\Omega, g) \text{ and } |\nabla^g R_g|\in L^{\infty}(\Omega, g)$. 
\end{enumerate}
Then $\Phi$ is an embedding and $\partial \Phi(M) = \mathbb{S}^n\setminus\Phi(M)$ is a finite point set. 
\end{theorem-m}

On the organization of this paper, we discuss some preliminaries and background in Section \ref{Sec:preli}. We present proofs for 
Theorem \ref{Thm:intro-1} and Theorem \ref{Thm:intro-2} in Section \ref{Sec:domain}. The proof of Theorem \ref{Thm:intro-2} takes most of 
Section \ref{Sec:domain} and is divided into 4 steps. In Section \ref{Sec:injectivity} we follow quite closely from \cite{SY88} and prove
Theorem \ref{Thm:intro-3}, Theorem \ref{Thm:intro-4} and Proposition \ref{Prop:intro-1}. Finally we conclude Main Theorem
at the end of Section \ref{Sec:injectivity}.


\section{Preliminaries and background}\label{Sec:preli}

In this section we collect some basic definitions and facts that are relevant to us in this paper. We first recall the celebrated Huber Theorem
\cite{Hu57}. Then we introduce the n-Laplace equations in conformal geometry and recall our
previous work that extended Arsove-Huber Theorem \cite{AH73} to n-superharmonic functions in general dimensions. We
also recall definitions and basic properties of p-capacity and p-parabolicity for our use later \cite{AM73, SY94}.
 
\subsection{Gaussian curvature equations and Huber's Theorem}\label{Subsec:huber57}

On a surface $(M^2, \bar g)$ with Gaussian curvature $K$, let the Gaussian curvature for a conformal metric 
$$
g = e^{2u}\bar g
$$
be $K_g$. Then 
$$
-\Delta u + K  = K_g \, e^{2u}
$$
is the celebrated Gaussian curvature equation. Recall 

\begin{theorem} \label{Thm:huber1957} (Huber \cite{Hu57}) Suppose that $(M^2, g)$ is a complete open surface and that
$$
\int_M (K^- dvol)_g < \infty,
$$
where $K^-_g$ is the negative part of Gaussian curvature of the metric $g$. Then $M$ is a compact surface with finitely many points
removed and $(M^2, g)$ is parabolic.

\end{theorem}

Huber in \cite{Hu57} extended the works of Cohn-Vossen \cite{CV35} and marked the highlight of applications of linear potential 
theory in differential geometry of surfaces then.


\subsection{n-Laplace equations in conformal geometry and n-superharmonic functions}\label{Subsec:n-superharmonic}

Suppose that $(M^n, \bar g)$ is a Riemannian manifold and that $g=e^{2u} \bar g$ is a conformal metric. 
In an attempt to search for Huber's type theorem in higher dimensions in conformal geometry, we recall the formula for the transformation 
of Ricci curvature tensor under conformal change:
$$
 Ric_{ij} - (n-2) u_{i,j}+ (n-2)u_{i}u_{j} - (\Delta u)\bar g_{ij}-(n-2)|\nabla u|^{2} \bar g_{ij} = (Ric_g)_{ij},
$$
where geometric quantities with a subscript $g$ are taken with respect to the metric $g$, while those without a subscription are 
taken with respect to the background
metric $\bar g$. Multiplying $|\nabla u|^{n-4}$ and contracting with $\nabla u$ on both sides, one gets
$$
\aligned
-((n-2) |\nabla u|^{n-4} u_iu_j u_{i, j} & + |\nabla u|^{n-2}\Delta u)  + |\nabla u|^{n-4} Ric (\nabla u, \nabla u) \\
& = |\nabla u|^{n-4}Ric_g (\nabla u, \nabla u)
\endaligned
$$
which is
\begin{equation}\label{Equ:n-laplace}
-  \Delta_n u +  |\nabla u|^{n-2} Ric (\nabla u)  = (|\nabla u|^{n-2} Ric (\nabla u))_g \, e^{nu}
\end{equation}
where $Ric(v)$ is the Ricci curvature in the direction of $v$. The right hand side $(|\nabla u|^{n-2}Ric(\nabla u))_g$ is 
calculated for the conformal metric $g = e^{2u}\bar g$. This n-Laplace equation reduces to Gaussian curvature equation in dimension 2
and therefore is a natural extension of Gaussian curvature equation in higher dimensions. Applications of n-Laplace equations \eqref{Equ:n-laplace}
have been initiated and studied in \cite{BMQ17, MQ18}. By the theory of Wolff potential in relation to n-Laplace equations developed in, for instance, 
\cite{KM94} and references therein, in \cite{MQ18}, we obtained

\begin{theorem} \label{Thm:MQ18} (Ma-Qing \cite[Theorem 1.1]{MQ18})  
Let $w$ be a nonnegative lower semi-continuous function that is $n$-superharmonic in $B(0, 2)\subset \r^n$ and 
\[
-\Delta_n u=\mu\ge 0
\]
for a Radon measure $\mu\ge 0$ and the origin is in the support $\mathcal{S}$ of the Radon measure $\mu$. 
Then there is a set $E\subset\mathbb{R}^n$, which is $n$-thin at the origin, such that
\[
\lim_{x\notin E \text{ and } |x|\rightarrow 0}\frac{u(x)}{\log\frac{1}{|x|}}=\liminf_{|x|\rightarrow0}\frac{u(x)}{\log\frac{1}{|x|}}=m\ge 0
\]
and
\[
u(x)\geq m\log\frac{1}{|x|} - C \ \text{ for $x\in B(0, 1)\setminus \{0\}$ and some $C$}.
\]
Moreover, if $u\in C^2(B(0,2)\setminus \{0\})$ and $(B(0, 2)\setminus \{0\},e^{2u}|dx|^2)$ is complete at the origin, then $m\geq1$. 
\end{theorem}

This theorem generalizes Arsove-Huber theorem \cite{AH73}
from 2 dimensions to general dimensions for n-superharmonic functions. For the definitions of n-thinness readers are referred to
\cite{AM72, KM94, MQ18}, for example. In \cite{MQ18}, there are detailed discussions on the differences of the definition in \cite{MQ18} to 
the previous, even though both reduce to the same in 2 dimensions. In this paper, inspired by Huber's theorem (Theorem \ref{Thm:huber1957}), 
we want to study differential geometric properties in large for certain manifolds without assuming Ricci nonnegative. 
Theorem \ref{Thm:MQ18} is one of the important ingredients. 
\\

We also want to recall some result on solving n-Laplace equations when the inhomogeneous term is a nonnegative Radon measure.

\begin{theorem}\label{Thm:KM92} (Kilpelainen-Maly \cite[Theorem 2.4]{KM92})
Let $U$ be a bounded domain in $\mathbb{R}^n$ and $\mu$ be a nonnegative Radon measure on $U$. Then there is a nonnegative
n-superharmonic function $u$ such that
$$
-\Delta_n u = \mu
$$
and $\min\{u, k\}\in H^{1, n}_0(U)$ for all positive number $k$.
\end{theorem}


\subsection{p-capacity and p-parabolicity}\label{Subsec:capacity-parabolicity}

In this subsection we review materials about p-capacity and p-parabolicity. First, let us recall the definition of p-capacity.

\begin{definition}\label{Def:p-capacity} For a compact subset $K$ in a domain $\Omega$ in a Riemannian manifold $(M^n, g)$, 
we define, for $1<p  < \infty$,
$$
cap_p(K,  \Omega) = \inf \{\int_\Omega |\nabla u|^p dvol: u\in C^\infty_0(\Omega) \text{ and $u\geq 1$ on $K$}\}.
$$
\end{definition}

The following relation between the Hausdorff measure and the p-capacity is well known. 
Readers are referred to \cite{AM73} and references therein. 

\begin{lemma}\label{Lem:n-cap-Hausdorff} (\cite[Theorem 2.10, Chapter VI]{SY94}) Let $K$ be a compact subset of the Euclidean space 
$\mathbb{R}^n$, and $1 < q \leq n$. We know 
\begin{enumerate}
\item if $H_{n-q}(K) < \infty$ then $cap_q(K) = 0$;
\item $if H_{n-q +\epsilon}(K) > 0$ for some $\epsilon > 0$ then $cap_q(K) > 0$.
\end{enumerate}
\end{lemma}

For our purpose, we adopt the following definition of p-parabolicity for noncompact manifolds. 

\begin{definition}\label{Def:p-parabolicity}
Suppose that $(M^n, g)$ is a complete noncompact Riemannian manifold. $(M^n, g)$ is said to be p-parabolic if
$cap_p(\bar U, M^n) = 0$ for any bounded domain $U\subset M$.
\end{definition}

This definition of p-parabolicity is known to be equivalent to the more conventional definition, that is, a complete noncompact manifold
$(M^n, g)$ is p-parabolic if there is no positive Green function for p-Laplacian $\Delta_p$ on $M$. Readers are referred to \cite{Hol90,
Tro99} and \cite[page 13]{CHS00}. From Definition \ref{Def:p-parabolicity}, 
it is easily seen that the p-parabolicity is quasi-isometrically invariant, that is, if $C^{-1}g_{2}\le 
g_{1}\le Cg_{2}$, then $(M,g_{1})$ is p-parabolic if and only if $(M,g_{2})$ is. It turns out that $n$-parabilicity is actually conformally
invariant, since n-capacity is conformally invariant. It is well known that volume growth is closely related to p-parabolicity. We recall, for instance, 
the following reference \cite[Corollary 3.2]{CHS00}. 

\begin{lemma}\label{Lem:volume-parabolicity} (Coulhon-Holopainen-Saloff-Coste \cite[Corollary 3.2]{CHS00})
Suppose that $(M^n, g)$ is a complete manifold. for $1 < p < \infty$, assume that, for a point $x_0\in M$,  
$$
\liminf_{r\to\infty} r^{-p} Vol(B(x_0, r))  = 0 \text{ or } Vol(B(x_0, r)\leq Cr^p(\log r)^{p-1}.
$$
Then $(M^n, g)$ is p-parabolic.
\end{lemma}

We also want to mention a volume growth estimate from integral lower bound of Ricci curvature.

\begin{lemma}\label{Lem:gallot}(Gallot \cite[Theorem 1]{Gal88}) Suppose that $(M^n, g)$ is a complete, noncompact manifold with infinite volume.
For $p\in (n, \infty]$, assume $Ric^-_g\in L^\frac p2(M^n, g)$. Then, for any point $x_0\in M^n$, 
\begin{equation}\label{Equ:gallot}
\lim_{r\to\infty} \frac {Vol(B(x_0, r))^\frac 1p}r = 0.
\end{equation}
\end{lemma}

Recently the critical case when $p = n$ has been resolved with the additional and appropriate assumption in \cite[Theorem 2.1]{Car20}. 
Consequently from Lemma \ref{Lem:gallot}, we observe 

\begin{lemma}\label{Lem:ricci-parabolicity}Suppose that $(M^n, g)$ is a complete, noncompact manifold.
Assume $Ric^-_g\in L^\frac p2(M^n, g)$ for some $p\in (n, \infty]$.  Then $(M^n, g)$ is p-parabolic
\end{lemma}

It turns out that n-parabolicity is out of reach in this approach, even though n-parabolicity is more favorable in conformal geometry. 


\section{Finite point compactifications for domains in $\mathbb{S}^n$}\label{Sec:domain}

In this section we discuss extensions of Huber's type theorem in general dimensions using n-Laplace equations \eqref{Equ:n-laplace}. First we derive a 
Huber's type theorem in the line of Carron-Herzlich \cite{CH02}. We will study examples to compare our Theorem \ref{Thm:theorem-A} to 
\cite[Theorem 2.1]{CH02}. To derive a finite point compactification we focus our attention to domains in the round
sphere $\mathbb{S}^n$ based on our earlier work \cite{MQ18}.  
 
 
\subsection{Domains in a compact manifold} \label{Subsec:domains}
In this subsection we state and prove a Huber's type theorem for domains in a given compact Riemannian manifold. We will study examples where one 
could not improve Theorem \ref{Thm:theorem-A} as it is stated. Theorem \ref{Thm:theorem-A} and Corollary \ref{Cor:CHY-geometric} should be 
compared with \cite[Corollary 2.2]{CH02}. Instead of the volume growth assumptions, we will use the integral lower bound on Ricci curvature. 
For this purpose, we let $Ric^-$ be the negative part of the smallest Ricci curvature.  

\begin{theorem}\label{Thm:theorem-A}
For $n\geq 3$, let $\Omega$ be a domain in a compact manifold $(M^n, \bar g)$ and endowed with a conformal metric 
$g = e^{2u}\bar g$. Assume 
\begin{equation*}
\lim_{x\to\partial\Omega}u(x)=+\infty \text{ and } \left\{
\aligned
 & \text{either } Ric^{-}_g \in L^{\frac{n}{2}}(\Omega,g),\\
 & \text{or } Ric^{-}_g |\nabla u|^{n-2}e^{2u} \in L^{1}(\Omega, \bar g).
\endaligned\right.
\end{equation*}
Then $cap_{n}(\partial\Omega,D)=0$, where $D$ is a neighborhood
of $\partial\Omega$ in $M$. Hence $(\Omega, g)$ is $n$-parabolic and $\partial\Omega$
is of Hausdorff dimension 0 in $(M^n, \bar g)$. 
\end{theorem}

\proof The proof is based on a rather straightforward estimate of n-capacity when test functions are chosen right. 
To construct appropriate test functions, for $\alpha$ large and fixed and $\beta$ large, we consider the cut-off 
$$
u_{\alpha, \beta} = \left\{\aligned \beta &\text{ when $u-\alpha \geq \beta$}\\
u-\alpha &\text{ when $u-\alpha <\beta$}\endaligned\right.
\quad\text{ and } \quad \phi = u_{\alpha, \beta} - \beta\eta
$$ 
where $\eta\in C^\infty_0 (\Sigma_\alpha)$ is a fixed usual cut-off that is equal to one in a neighborhood of $\partial\Omega$
and $\Sigma_\lambda = \{x\in \Omega: u(x) > \lambda\}$ is the level set of the conformal factor function $u$. It is easily seen that
$$
u_{\alpha,\beta} \in (0, \beta) \text{ in $\Sigma_\alpha$ and }  \phi  = 0 \text{ on } \left\{\aligned & \{x\in \Omega: u= \alpha\} \text{ or} \\
                                                                   & \{x\in \Omega: u \geq \alpha +\beta\} \endaligned\right.
$$ 
for $\beta$ large. We multiple $\phi$ to both sides of \eqref{Equ:n-laplace} and integrate 
$$
\int_{\Sigma_\alpha} |\nabla u|^{n-2} \nabla u \cdot\nabla\phi + \int_{\Sigma_\alpha} |\nabla u|^{n-2}Ric(\nabla u)\phi 
=  \int_{\Sigma_\alpha} |\nabla u|^{n-2} Ric_g (\nabla u)e^{2u} \phi.
$$ 
Therefore 
$$
\aligned 
 \int_{\Sigma_\alpha} |\nabla u_{\alpha, \beta}|^n  & \leq \beta \int_{\Sigma_\alpha}|\nabla u|^{n-2} \nabla u\cdot \nabla\eta
-  \int_{\Sigma_\alpha}|\nabla u|^{n-2} Ric(\nabla u) \phi \\
& \quad + \beta \int_{\Sigma_\alpha}|\nabla u|^{n-2} Ric_g(\nabla u)e^{2u} (1-\eta) \\
 & \quad  + \int_{\Sigma_\alpha} |\nabla u|^{n-2} Ric_g(\nabla u) e^{2u} (u_{\alpha, \beta}- \beta).
\endaligned
$$
Because 
$$
\phi = u_{\alpha, \beta} - \beta \eta = u_{\alpha, \beta} - \beta + \beta (1-\eta).
$$
Notice that $\eta$ is independent of $\beta$, which implies the terms except the last term in the right hand side is bounded 
by $C\beta$. In the mean time, the last term in the right
\begin{equation}\label{Equ:bad-term}
 \int_{\Sigma_\alpha} |\nabla u|^{n-2} Ric_g(\nabla u) e^{2u} (u_{\alpha, \beta}- \beta) \leq \beta 
 \int_{\Sigma_\alpha} Ric^-_g |\nabla u|^{n-2} e^{2u} dvol \leq C\beta
\end{equation}
if we use the assumption $Ric^-_g|\nabla u|^{n-2}e^{2u} \in L^1(\Omega, \bar g)$. Otherwise, we use the alternative assumption
and apply H\"{o}Lder inequality in \eqref{Equ:bad-term}
$$
\aligned
\int_{\Sigma_\alpha} |\nabla u|^{n-2} Ric_g(\nabla u) e^{2u} (u_{\alpha, \beta}- \beta) & \leq 
 C\beta (\int_{\Sigma_\alpha} |Ric^-_g|^\frac n2 e^{nu})^\frac 2n (\int_{\Sigma_\alpha} |\nabla u_{\alpha,\beta}|^n)^{1 - \frac 2n}\\
 & \leq \frac 12 \int_{\Sigma_\alpha}|\nabla u_{\alpha,\beta}|^n + C \beta^\frac n2 \int_{\Sigma_\alpha} (|Ric^-|^n\, dvol)_g
\endaligned
 $$
Thus, in both cases, 
$$
\int_{\Sigma_\alpha} |\nabla (\frac {u_{\alpha, \beta}}\beta)|^n \leq C \beta^{1-n} + C\beta^{-\frac n2} \text{ as $\beta\to\infty$}
$$
which complete the proof of this theorem.
\endproof

Next we state and prove an improvement of  \cite[Proposition 8.1]{CHY04}.

\begin{lemma}\label{Lem:CYH-improved}
Suppose that $(M^n, \bar g)$ is a compact manifold and that $\Omega$ is a domain in $M$. Let $g = e^{2w} \bar g$ be a complete
conformal metric on $\Omega$. Assume that $R^-_g\in L^p(\Omega, g)$ for some $p> \frac n2$. Then
\begin{equation}\label{Equ:CHY-improved}
w(x) \to \infty \text{ as $x\to\partial\Omega$}.
\end{equation}
\end{lemma}

\begin{proof}
The proof of \eqref{Equ:CHY-improved} is inspired by the Moser's iteration argument of \cite[Proposition 8.1]{CHY04} applied to the scalar curvature 
equation
\begin{equation}\label{Equ:scalar curvature}
-\Delta_g u + \frac {n-2}{4(n-1)}R^+_g u = \frac {n-2}{4(n-1)}R^-_g u +  \frac {n-2}{4(n-1)}R_{\bar g} u^\frac {n+2}{n-2}
\end{equation}
for $u = e^{- \frac {n-2}2 w}$. To get ride of the nonlinear term in the right hand side, without the loss of generality, we may assume 
\begin{enumerate}
\item $R_{\bar g} \leq 0$ when $(M^n, \bar g)$ is of nonpositive Yamabe type;
\item $R_{\bar g} = 0$ in an open domain $D$ that includes $\partial\Omega$ in $M$ when $(M^n, \bar g)$ is of positive Yamabe type.
\end{enumerate}
(2) in the above can be achieved by using the Green function of the conformal Laplacian $-\Delta + \frac {n-2}{4(n-1)}R$ with a pole outside of 
the domain $D$. Now \eqref{Equ:scalar curvature} becomes
\begin{equation}\label{Equ:moser}
-\Delta_g u + \frac {n-2}{4(n-1)}R^+_g u \leq \frac {n-2}{4(n-1)}R^-_g u. 
\end{equation}
To apply the Moser's iteration argument, we need the Sobolev inequality. In fact, from
\cite[Proposition 8.1]{CHY04}, what we really need is the Yamabe-Sobolev \eqref{Equ:Yamabe-Sobolev} in below. 
After we have the Yamabe-Sobolev \eqref{Equ:Yamabe-Sobolev},  the difference from 
\cite[Proposition 8.1]{CHY04} is that we have $R^{-}_g \in L^{p}(\Omega, g)$ instead of  $R_g$ bounded from below. For this, one simply needs 
to observe that $R^{-}_g \in L^{p}(\Omega, g)$ for $\frac n2 < p\leq \infty$ is enough for Moser's iteration in \eqref{Equ:moser} (cf.  
(8.5) of \cite[Proposition 8.1]{CHY04}).
\\

Therefore this lemma follows from the following Yamabe-Soboev
\begin{equation}\label{Equ:Yamabe-Sobolev}
(\int v^\frac {2n}{n-2}dvol_g)^\frac {n-2}n \leq C (D) \int (|\nabla v|^2 + \frac {n-2}{4(n-1)} R_g v^2)dvol_g \text{ for all $v\in 
C^\infty_c(D\setminus\partial\Omega)$}
\end{equation}
for some open neighborhood $D$ of $\partial\Omega$ in $M$. To prove \eqref{Equ:Yamabe-Sobolev} when $(M^n,, \bar g)$ is of nonpositive 
Yamabe type, i. e. $R_{\bar g}\leq 0$ on $M$ (cf. (1) in the above), without the loss of generality, for the compact manifold $(M^n, \bar g)$, we may assume
\begin{itemize}
\item $R_{\bar g} \geq - c_1$ on $M$ for some nonnegative constant $c_1$;
\item the Sobolev inequality 
\begin{equation}\label{Equ:Sobolev-g bar}
(\int v^\frac {2n}{n-2}dvol_{\bar g})^\frac {n-2}n \leq S (M^n, \bar g) \int (|\nabla v|^2 dvol)_{\bar g} \text{ for all $v\in C^\infty_c(M)$}.
\end{equation}
\end{itemize}
By the definition of Yamabe constant, it is equivalent to prove, for the metric $\bar g$, the following Yamabe-Sobolev holds
\begin{equation}\label{Equ:Y-S-g bar}
(\int v^\frac {2n}{n-2}dvol_{\bar g})^\frac {n-2}n \leq C_D\int ((|\nabla v|^2 + \frac {n-2}{4(n-1)} R v^2)dvol)_{\bar g} \text{ for all 
$v\in C^\infty_c(D\setminus\partial\Omega)$}
\end{equation}
for some open neighborhood $D$ of $\partial\Omega$ in $M$. To do so, one only needs to control
$$
\aligned
-\frac {n-2}{4(n-1)}\int R_{\bar g} v^2 dvol_{\bar g}  & \leq c_1 \frac {n-2}{4(n-1)} \int v^2 dvol_{\bar g} \\
& \leq c_1 \frac {n-2}{4(n-1)} (\int v^\frac {2n}{n-2} dvol_{\bar g})^\frac {n-2}n Vol(D\setminus\partial\Omega, \bar g)^\frac 2n \\
& \leq c_1 \frac {n-2}{4(n-1)} Vol(D\setminus\partial\Omega, \bar g)^\frac 2n S(M, \bar g)\int (|\nabla v|^2dvol)_{\bar g},
\endaligned
$$
which implies \eqref{Equ:Y-S-g bar} from \eqref{Equ:Sobolev-g bar}, provided that
\begin{equation}\label{Equ:volume}
Vol(D\setminus\partial\Omega, \bar g)^\frac 2n < \frac {2(n-1)}{c_1 (n-2)S(M, \bar g)}.
\end{equation}
Because 
$$
\aligned
\int (|\nabla v|^2 dvol)_{\bar g} & =  \int ((|\nabla v|^2 + \frac {n-2}{4(n-1)} R v^2)dvol)_{\bar g} -\frac {n-2}{4(n-1)}\int R_{\bar g} v^2 dvol_{\bar g} \\
&\leq  \int ((|\nabla v|^2 + \frac {n-2}{4(n-1)} R v^2)dvol)_{\bar g} + \frac 12 \int (|\nabla v|^2 dvol)_{\bar g}.
\endaligned
$$
Notice that \eqref{Equ:volume} holds for some domain $D$ regardless of the Lebesgue measure of $\partial\Omega$ is positive or not. 
\\

On the other hand, when $(M^n,, \bar g)$ is of positive Yamabe type, i. e. $R_{\bar g} = 0$ in $D\subset M$ (cf. (2) in the above), 
\eqref{Equ:Y-S-g bar} is the same as \eqref{Equ:Sobolev-g bar} and implies \eqref{Equ:Yamabe-Sobolev}.
\end{proof}

\begin{corollary}\label{Cor:CHY-geometric}
For $n\geq 3$, let $\Omega$ be a domain in a compact manifold $(M^n, \bar g)$ and endowed with a complete conformal metric 
$g = e^{2u}\bar g$. Assume 
Assume that,  
\begin{equation*}
R^-_g \in L^p(\Omega, g) \text{ for some $p\in(\frac{n}{2},+\infty]$ and } \left\{
\aligned
 & \text{either } Ric^{-}_g \in L^{\frac{n}{2}}(\Omega, g),\\
 & \text{or } Ric^{-}_g |\nabla u|^{n-2}e^{2u} \in L^{1}(\Omega, \bar g).
\endaligned\right.
\end{equation*}
where $R^-_g$ is the negative part of the scalar curvature $R$. Then $cap_{n}(\partial\Omega,D)=0$, where $D$ is a neighborhood
of $\partial\Omega$ in $M^n$. Hence $(\Omega, g)$ is $n$-parabolic and $\partial\Omega$
is of Hausdorff dimension 0 in $(M^n, \bar g)$. 
\end{corollary}

In the following we give an example to illustrate that it is impossible to push for $\partial\Omega$ to be a finite point set in 
Theorem \ref{Thm:theorem-A}, more precisely, not under the assumption $Ric^-_g\in L^\frac n2(\Omega, g)$ (also see a similar example 
\cite[Theorem 2.7]{Car20}). The idea of such 
a construction comes from the paper of Gallot \cite{Gal88}. We refer the readers to Appendix A1 and A2 of \cite{Gal88} for details. 
We are constructing a complete conformal metric $g = e^{2u}g_{\mathbb{S}}$ on a domain $\Omega\subset\mathbb{S}^n$ such that
\[
\int_{\Omega} |Ric^{-}_g|^{n/2}dvol_g <+\infty
\] 
but $\partial\Omega$ consists of infinite many points. For an easy visual presentation, we consider $\Omega = \Phi(M)$ for a conformal embedding
$$
\Phi: M \to \mathbb{S}^n.
$$
We first construct the manifold $(M^n, g)$ and then the conformal embedding will be easily seen. Let 
\[
C=[-1,1]\times\mathbb{S}^{n-1},g_{\alpha,\varepsilon,\eta}=dt^{2}+b^{2}(t)g_{\mathbb{S}^{n-1}},
\]
where $b(t)=\eta(t^{2}+\varepsilon^{2})^{\frac{\alpha}{2}}.$ For
$k\ge1$, we choose
\[
\alpha=\alpha_{k}=1+k\cdot e^{-k},\varepsilon=\varepsilon_{k}=e^{-e^{k}},\eta=\eta_{k}=\frac{1}{\alpha_{k}}(1+\varepsilon_{k}^{2})^{\frac{2-\alpha_{k}}{2}}.
\]
We denote $g_{\alpha_{k},\varepsilon_{k},\eta_{k}}$ by $g_{k}$ and
$(C,g_{k})$ by $C_{k}$. The manifold $M$ is obtained by gluing
countable many $B_k$, which is the Euclidean space with a ball removed, to $A$, which is the Euclidean space with countably many balls removed, 
using $C_k$ to form a complete noncompact manifold. The following figure indicates what is $M$ and what is the image $\Phi(M)$ in $\mathbb{S}^n$.
From an argument similar to that in \cite[A2]{Gal88} we have
\[
\int_{M}|Ric^{-}_g|^{\frac{n}{2}}dvol_g\le C\sum_{k}[(k\cdot e^{-k})^{\frac{n-2}{2}}+e^{-(n-1)k}]<+\infty.
\]
We therefore conclude that $\Omega=\mathbb{S}^{n}\backslash\{{\rm countable\,\,many\,\,points}\}.$
\begin{center}
\includegraphics[scale=0.43]{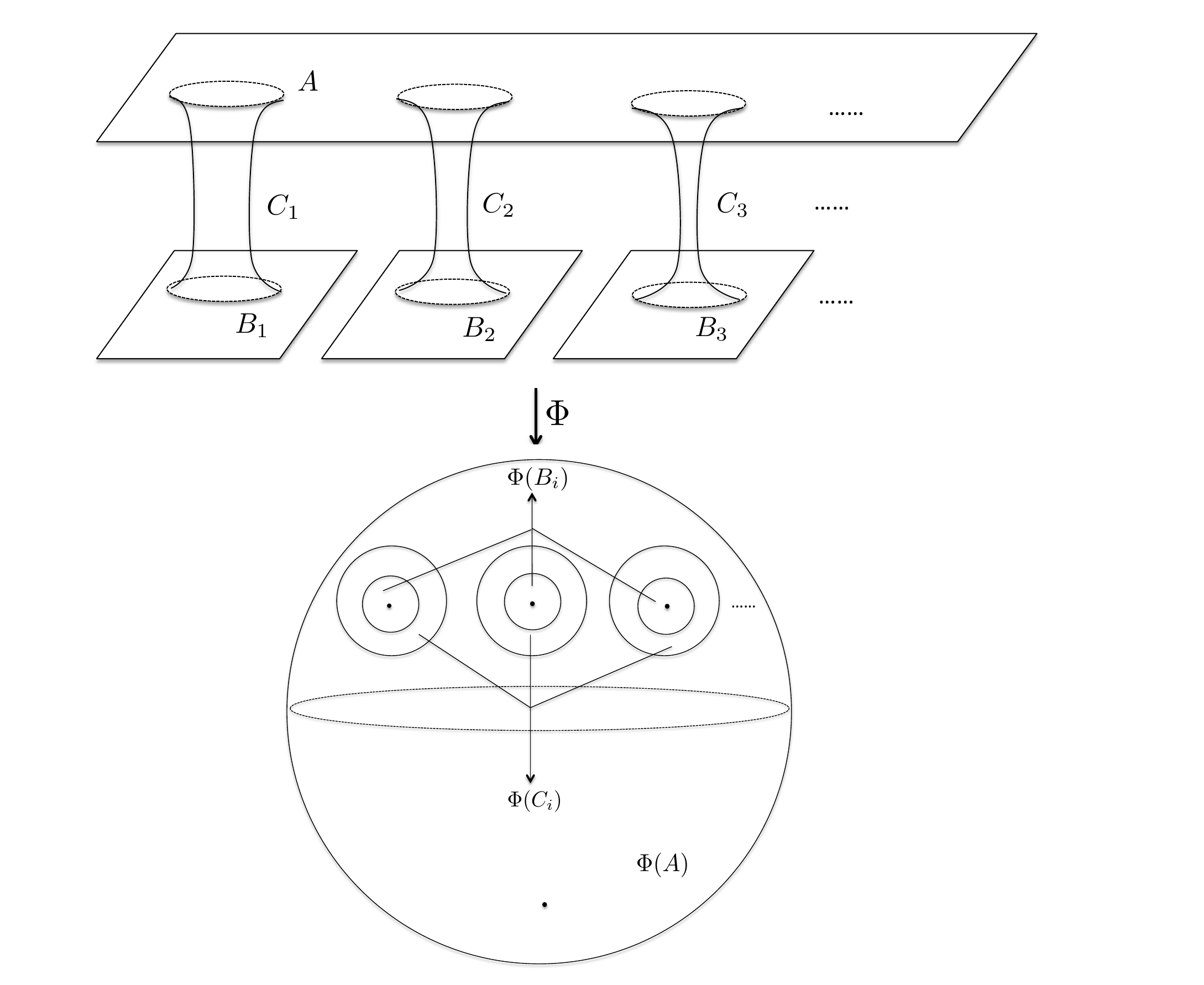}
\end{center}


\subsection{Finite point conformal compactification for domains in $\mathbb{S}^n$}\label{Subsec:finite point}

In this subsection we will study domains in the standard round sphere $\mathbb{S}^n$. Using n-Laplace equations and what have
been developed in \cite{MQ18} we state and prove a Huber's type finite point compactification theorem for domains in $\mathbb{S}^n$.

\begin{theorem}\label{Thm:theorem-B}
For $n\geq 3$, let $\Omega$ be a domain in the standard unit round sphere $(\mathbb{S}^n, g_{\mathbb{S}})$.
Suppose that, on $\Omega$, there is a conformal metric $g= e^{2u}g_{\mathbb{S}}$ satisfying
$$
\lim_{x\to\partial\Omega}u(x)=+\infty \text{ and } 
Ric^{-}_g|\nabla u|^{n-2}e^{2u} \in L^{1}(\Omega,g_{\mathbb{S}}).
$$
Then $\partial\Omega = \mathbb{S}^n\setminus\Omega$ is a finite point set. 
\end{theorem}

\proof Our proof of this theorem is divided into the following four steps. But, first, in order to employ estimates of Wolff potential, 
we consider n-Laplace equations on Euclidean space $\mathbb{R}^n$ via a stereographic projection $\mathcal{P}$ 
with respect to an interior point of $\Omega$. 
We have then a complete conformal metric $g = e^{2v}|dx|^2$ and n-Laplace equations
\begin{equation}\label{Equ:euclidean-n-laplace}
-\Delta_n v = Ric_g(\nabla v) |\nabla v|^{n-2} e^{2v} = f(x) \qquad  \text{ in $\mathcal{P}(\Omega) \subset \mathbb{R}^n$}
\end{equation}
thanks to the conformal invariance of the n-Laplace equation \eqref{Equ:n-laplace}. 
$\mathcal{S} = \partial\mathcal{P}(\Omega)\subset \mathbb{R}^n$ becomes a compact subset of zero Hausdorff 
dimensions in the light of Theorem \ref{Thm:theorem-A}. To focus on the asymptotic behavior near $\mathcal{S}$, 
we will consider a bounded open neighborhood $D\subset \mathbb{R}^n$ 
of $\mathcal{S}$. The right hand side of \eqref{Equ:euclidean-n-laplace} satisfies that 
$f^- \in L^1(D, |dx|^2)$ as well as 
$$
\lim_{x\to\mathcal{S}} v(x) = \infty
$$
according to the assumptions. 
 
\vskip 0.1in
\noindent
{\bf Step 1.} In this step we show that in fact 
$$
f\in L^1(D, |dx|^2)
$$
from $f^-\in L^1(D, |dx|^2)$. This situation is similar to what was dealt with in \cite[Lemma 1]{BV89}(see also \cite{KV86, V17}), particularly 
the idea in \cite[Lemma 1]{BV89} seems to work. But we adopt a more straightforward argument here. 
The following cutoff function is very useful to our investigations and has been used in \cite[Lemma 14]{DHM97} and \cite{MQ18}. Let
\[
\alpha_{s}(t)=\begin{cases}
t & t\le s,\\
{\rm increasing} & s\le t\le10s \\
2s & t\ge10s.
\end{cases}
\]
We may ask $0\leq \alpha_{s}'(t)\le1$ and $\alpha''_{s}(t)\le0$. It is easily seen that
\begin{itemize}
\item $\alpha_s(t) \leq t$
\item $\lim_{s\to+\infty}\alpha_{s}(t)=t \text{ for any $t\in (-\infty, \infty)$}$
\item $\alpha_s(v) = 2s$ in the neighborhood $\{x\in\mathbb{R}^n: v(x) > 10s\}$ of $\partial D$
\item $-\Delta_{n}\alpha_{s}(v)=0$ in the neighborhood $\{x\in\mathbb{R}^n: v(x) > 10s\}$ of $\partial D$
\item $\alpha_{s}(v)\in C^{\infty}(\mathbb{R}^n)$.
\end{itemize}
More importantly,
\begin{equation} \label{Equ:inside}
\int_{D}-\Delta_{n}\alpha_{s}(v) = \int_{\partial D} |\nabla v|^{n-2} \frac{\partial v}{\partial\nu} 
\end{equation}
provided $s>\max_{\partial D}v$. Let
\begin{equation}\label{Equ:f_s}
f_{s}=-\Delta_{n}a_{s}(v).
\end{equation}
We shall prove that 
\begin{equation}\label{L1 norm of delta as}
\|f_{s}\|_{L^{1}(D, |dx|^2)}\le C \text{ for a constant $C$ that is independent of $s$}.
\end{equation}
We first calculate 
\begin{equation}\label{Equ:laplace of alpha-s}
- \Delta_n \alpha_s(v)  = (\alpha'_s(v))^{n-1} (-\Delta_n v) - (n-1) (\alpha'_s(v))^{n-2} \alpha''_s(v) |\nabla v|^n,
\end{equation} 
where the second term in the right side $-(n-1) (\alpha'_s(v))^{n-2} \alpha''_s(v) |\nabla v|^n$ is nonnegative because 
$\alpha_s'' \leq 0$. Hence
$$
f_s^- = ((\alpha'_s(v))^{n-1} (-\Delta_n v) - (n-1) (\alpha'_s(v))^{n-2} \alpha''_s(v) |\nabla v|^n)^- \leq 
(\alpha'_s(v))^{n-1} f^- \leq f^-,
$$
where we use the fact that $\alpha_s'\in [0, 1]$. This implies
$$
 \|f_s^- \|_{L^1(D, |dx|^2)} =\int_D f_s^- dx \leq \|f^- \|_{L^1(D, |dx|^2)}.
$$
Therefore
$$
\aligned
\int_D f_s^+ dx & = \int_D f_s  dx + \int_D f_s^- dx  \\
& = \int_{\partial D} |\nabla v|^{n-2} \frac{\partial v}{\partial\nu} 
+ \int_D f_s^- dx \\
& \leq \int_{\partial D} |\nabla v|^{n-2} \frac{\partial v}{\partial\nu} +  \|f^-\|_{L^1(D, |dx|^2)}.
\endaligned
$$
The claim \eqref{L1 norm of delta as} is proven. Next we apply Fatou lemma
$$
\int_D f^+ dx \leq \liminf_{s\to \infty} \int_D f_s^+ \leq \int_{\partial D} |\nabla v|^{n-2} \frac{\partial v}{\partial\nu} + \|f^-\|_{L^1(D, |dx|^2)}
$$
and complete the first step. We now collect what we have got in this step as a lemma.

\begin{lemma}\label{Lem:B-step-1}
Under the assumptions in Theorem \ref{Thm:theorem-B}, 
$$
\aligned
\|f_s\|_{L^1(D, |dx|^2)} & \leq \int_{\partial D} |\nabla v|^{n-2} \frac{\partial v}{\partial\nu} + 2 \|f^-\|_{L^1(D, |dx|^2)} \text{ and } \\
\|f\|_{L^1(D, |dx|^2)} & \leq \int_{\partial D} |\nabla v|^{n-2} \frac{\partial v}{\partial\nu} + 2 \|f^-\|_{L^1(D, |dx|^2)}.
\endaligned
$$
Moreover $-\Delta_n \alpha_s(v)$ is a Radon measure on $D$ and
$$
\aligned
|(-\Delta_n\alpha_s(v)) (\phi) |& = |\int_{\mathbb{R}^n} (-\Delta_n\alpha_s(v))\phi dx| \\
& \leq (\int_{\partial D} |\nabla v|^{n-2} \frac{\partial v}{\partial\nu} + 2\|f^-\|_{L^1(D, |dx|^2)}) \|\phi\|_{C^0(D)}
\endaligned
$$
for any continuous function with compact support $\phi\in C_c^0(D)$, provided  $s$ is appropriately large. 
\end{lemma}

\vskip 0.1in
\noindent{\bf Step 2.} In this step we want to show that in fact $-\Delta_n v$ is a Radon measure in $D$. The analysis in this step 
is very subtle and one needs to be very careful to understand the Radon measure $-\Delta_n v$ from 
$-\Delta_n\alpha_s(v)$. We start with $\phi\in C^\infty_c(D)$
$$
(-\Delta_n v )(\phi) = (|\nabla v|^{n-2}\nabla v) (\nabla\phi)
$$
by the definition of distributional derivatives. We then claim that $|\nabla v|^{n-1}\in L^1(D, |dx|^2)$ 
(cf. Lemma \ref{Lem:L^{n-1}}). Notice that $\nabla \alpha_s(v) = \alpha_s'(v) \nabla v$ and $\alpha_s'(v) \in [0, 1]$. 
Therefore, by the dominated convergence theorem, we have
\begin{equation}\label{Equ:DCT}
\aligned
(-\Delta_n v )(\phi) & = \int  |\nabla v|^{n-2}\nabla v \cdot\nabla\phi  \\
& = \lim_{s\to\infty} \int  |\nabla \alpha_s(v)|^{n-2}\nabla \alpha_s(v)\cdot\nabla\phi \\
& = \lim_{s\to\infty}(-\Delta_n\alpha_s(v))(\phi).
\endaligned
\end{equation}
Therefore, by Lemma \ref{Lem:B-step-1},
$$
|(-\Delta_n v )(\phi)| = |\lim_{s\to\infty} \int_D -\Delta_n\alpha_s(v) \phi dx| \leq  2\|f^-\|_{L^1(D, |dx|^2)} \|\phi\|_{C^0(D)},
$$
which implies that $-\Delta_n v$ is a Radon measure on $D$. 
In fact we want to be a bit more careful and calculate the following according to \eqref{Equ:laplace of alpha-s} and Step 1
$$
\aligned
\int_D (-\Delta_n v) \phi dx + & (-\Delta_n v)|_{\mathcal{S}} (\phi)  = (-\Delta_n v )(\phi) = \lim_{s\to\infty} \int_D -\Delta_n\alpha_s(v) \phi dx\\
& = \lim_{s\to\infty} \int_D [(\alpha'_s(v))^{n-1} (-\Delta_n v) - (n-1) (\alpha'_s(v))^{n-2} \alpha''_s(v) |\nabla v|^n ]dx \\
& \geq \lim_{s\to\infty} \int_D (\alpha'_s(v))^{n-1} (-\Delta_n v) dx \\
&\geq \int_D (-\Delta_n v)\phi dx,
\endaligned
$$
which implies that $-\Delta_n v$ is nonnegative on its support $\mathcal{S}$. Let us first collect what we have derived from the claim in this 
step as the following lemma.

\begin{lemma}\label{Lem:radon}
Under the assumptions in Theorem \ref{Thm:theorem-B}, $-\Delta_n v$ is a Radon measure on $\mathbb{R}^n$ and is nonnegative 
on its support $\partial D$. 
\end{lemma}

To complete the proof of Lemma \ref{Lem:radon}, we state a lemma to ensure the claim we made in the above. This lemma is well known 
for n-superharmonic functions (cf. \cite{KM92} and references therein). We want to extend it to our cases. 

\begin{lemma}\label{Lem:L^{n-1}} Under the assumptions in Theorem \ref{Thm:theorem-B}, for any $1\leq q< n$, 
$$
\|\nabla v\|_{L^q(D, |dx|^2)} \leq C \text{ and } \|\nabla \alpha_s(v)\|_{L^q(D, |dx|^2)} \leq C
$$
for a constant $C$ that is independent of $s$. 
\end{lemma}

\proof First we introduce the notations
$$
D_k = \{x\in \mathbb{R}^n: v(x) > k\}
$$
where we consider $v(x) = \infty$ on $\mathcal{S}$. We may assume that $D_k\subset D$ for $k=1, 2, \cdots$.
We start with the estimates using H\"{o}lder inequality
\begin{equation}\label{Equ:holder}
\aligned
\int_{D_l} |\nabla \alpha_s(v)|^q dx & = \int_{D_l} \alpha_s(v)^q |\alpha_s(v)^{-1}\nabla\alpha_s(v)|^q dx \\
& \leq (\int_{D_l} \alpha_s(v)^\frac {nq}{n-q}dx)^\frac {n-q}n (\int_{D_l} |\nabla\alpha_s(v)|^n \alpha_s(v)^{-n} dx)^\frac qn.
\endaligned
\end{equation}
For the first factor in the right hand side of \eqref{Equ:holder}, we rely on the Sobolev inequality.
$$
\aligned
\|\alpha_s(v)\|_{L^\frac {nq}{n-q}(D_l, |dx|^2)} & \leq \|\alpha_s(v) - l\|_{L^\frac {nq}{n-q}(D_l, |dx|^2)} + \| l\|_{L^\frac {nq}{n-q}(D_l, |dx|^2)}\\
& \leq C(n) \|\nabla\alpha_s(v)\|_{L^q(D_l, |dx|^2)} + l Vol(D_l)^{\frac 1q - \frac 1n},
\endaligned
$$ 
which implies
\begin{equation}\label{Equ:sobolev}
\|\alpha_s(v)\|_{L^\frac {nq}{n-q}(D_l)}^q \leq 2^{q-1}(C(n)^q\int_{D_l}|\nabla\alpha_s(v)|^qdx + l^q Vol(D_l)^{1 - \frac qn}).
\end{equation}
Next, we want to estimate the second factor in the right hand side of \eqref{Equ:holder}. Thanks to Theorem \ref{Thm:theorem-A}, we know
$$
\lim_{l\to\infty} cap_n(D_l, D_k) = 0
$$
for any fixed and large $k$. $k$ is chosen such that, for $\epsilon >0$, 
\begin{equation}\label{Equ:choice of k}
\frac 1{k^{n-1}}\int_{D_k}|f_s| dx \leq \frac 12 \epsilon^\frac nq.
\end{equation}
Then, for this fixed and large $k$, there exists $l$ and a test function $\zeta(x)$ such that
$$
\zeta (x)\in C^\infty_c(D_k), \zeta\in [0, 1], \zeta \equiv 1 \text{ on } D_l 
$$
and moreover 
\begin{equation}\label{Equ:choice of l}
\int |\nabla\zeta|^n dx \leq \frac 12 (\frac {n-2}{n-1})^{n-1} \epsilon^\frac nq.
\end{equation}
Multiplying $\zeta^n\alpha_s(v)^{-n+1}$ to both sides of \eqref{Equ:f_s} and integrating, we have
$$
\aligned
\int f_s \zeta^n \alpha_s(v)^{-n+1} dx & = \int (-\Delta_n\alpha_s(v)) \zeta^n\alpha_s(v)^{-n+1} dx \\
& = \int|\nabla\alpha_s(v)|^{n-2}\nabla\alpha_s(v)\cdot\nabla(\zeta^n\alpha_s(v)^{-n+1}) dx \\
= n \int |\nabla\alpha_s(v)|^{n-2} \alpha_s(v)^{-n+1} & \zeta^{n-1} \nabla\alpha_s(v) \cdot\nabla\zeta dx  -
(n-1)\int |\nabla\alpha_s(v)|^n \zeta^n\alpha_s(v)^{-n} dx,
\endaligned
$$
wich implies
$$
\aligned
(n-1)\int |\nabla\alpha_s(v)|^n & \zeta^n\alpha_s(v)^{-n} dx \leq \frac 1{k^{n-1}} \int_{D_k} |f_s|dx \\
+  (\frac{n}{n-1} (n-2)\int & |\nabla\alpha_s(v)|^n\zeta^n \alpha_s(v)^{-n} dx)^\frac {n-1}n  ( n(\frac {n-1}{n-2})^{n-1}\int |\nabla\zeta|^n dx)^\frac 1n.
\endaligned
$$
Therefore, applying Young's inequality, we get
\begin{equation}\label{Equ:young}
\aligned
\int |\nabla\alpha_s(v)|^n \zeta^n\alpha_s(v)^{-n} dx  & \leq \frac 1{k^{n-1}} \int_{D_k} |f_s|dx 
+  (\frac {n-1}{n-2})^{n-1} \int |\nabla\zeta|^n dx \\
& \leq \epsilon^\frac nq.
\endaligned
\end{equation}
in the light of \eqref{Equ:choice of k} and \eqref{Equ:choice of l}. Combining \eqref{Equ:holder}, \eqref{Equ:sobolev} and \eqref{Equ:young}, we arrive at
$$
\int_{D_l} |\nabla\alpha_s(v)|^q dx \leq 2^{q-1}C(n)^q\epsilon \int_{D_l}|\nabla\alpha_s(v)|^q dx + 
2^{q-1}l^q \epsilon Vol(D_l)^{1 - \frac qn} ,
$$
For appropriate choice of $\epsilon, k, l$ so that $2^{q-1}C(n)^q \epsilon \leq \frac 12$ (independent of $s$), the above implies
 \begin{equation}\label{Equ:integrability}
\int_{D_l} |\nabla\alpha_s(v)|^q dx \leq 2^{q}l^q \epsilon Vol(D_l)^{1 - \frac qn} <\infty
\end{equation}
and finishes the proof of this lemma by Fatou lemma.
\endproof

\vskip 0.1in
\noindent{\bf Step 3.} In this step we want to build an n-superharmonic function as a sharp upper bound to $v$ so that the 
strengthened \cite[Theorem 1.1]{MQ18} is applicable. 
For this purpose we introduce the desired n-superharmonic function $\mathcal{A}(v)$ to be the solution to
\begin{equation}\label{Equ:definition of v^+}
\left\{\aligned
-\Delta_n \mathcal{A}(v) & = (-\Delta_n v)^+ \text{ in $D$}\\
\mathcal{A}(v) & = v \text{ on $\partial D$}.
\endaligned\right.
\end{equation}
For the existence and properties of $\mathcal{A}(v)$, we follow \cite{KM92}. Specifically we follow the idea in the proof of \cite[Theorem 2.4]{KM92} to find 
the n-superharmonic function $\mathcal{A}(v)$. From Step 2, we know $(-\Delta_n \alpha_s(v))$ as Radon measures converges to the Radon measure
$(-\Delta_n v)$, which implies $(-\Delta\alpha_s(v))^+$ converges to $(-\Delta_n v)^+$. Hence we consider
\begin{equation}\label{Equ:definition of alpha_s(v)^+}
\left\{\aligned
-\Delta_n \mathcal{A}_s(v) & = (-\Delta_n \alpha_s(v))^+ \text{ in $D$}\\
\mathcal{A}_s(v) & = \alpha_s(v) \text{ on $\partial D$}.
\endaligned\right.
\end{equation}
The existences and uniqueness of $\mathcal{A}_s(v)$ is due to \cite[Theorem 1 and Remark 1]{BG89}, for instance. Notice that we may let
$$
D = \{x\in \mathbb{S}^n: v (x) > 1\}.
$$
Hence $\alpha_s(v) = v = 1$ on $\partial D$ when $s$ is appropriately large. Regarding regularity, for example, 
thanks to \cite[Theorem 3.3]{Te14}, $\mathcal{A}_s(v)\in C^{1, \gamma}$. Readers are referred to \cite{BG89, Te14} for more references and
history. As in the proof \cite[Theorem 2.4]{KM92}, from \cite[Theorem 1.17]{KM92}, at least for a sequence $s\to\infty$, 
$\mathcal{A}_s(v)$ converges to the n-superharmonic function $\mathcal{A}(v)$ that satisfies \eqref{Equ:definition of v^+}.

\begin{lemma}\label{Lem:comparison}
For $s$ appropriately large, in $D$, 
$$
\mathcal{A}_s(v) + 1 \ge \alpha_{s}(v)
$$ 
where $\mathcal{A}_s(v)$ is the solution to \eqref{Equ:definition of alpha_s(v)^+}.
\end{lemma}

\proof For simplicity, we substitute $\alpha_s(v)$ by $a_s$ and $\mathcal{A}_s(v)+1$ by $b_s$.
Assume otherwise that there is a point $x_{0}\in D$ such that $b_s (x_0) < a_{s}(x_{0})$. 
Then we can find a domain $U \subset D$ such that $b_{s}  < a_{s}$ in $U$
and $b_{s} = a_s$ on $\partial U$. Obviously $\partial U$ is in the interior of $D$ 
because $\mathcal{A}_s(v) = \alpha_s(v)$ on $\partial D$. We now choose a nonnegative cutoff function 
$\phi_{\varepsilon}\in C_{0}^{\infty}(U)$ such that 
\[
\phi_{\varepsilon}(x)=\begin{cases}
1, & x\in U_{2\varepsilon}=\{x\in U; d(x, \partial U) \ge 2\varepsilon\},\\
0, & x\in U \backslash U_\varepsilon = \{x\in U: d(x, \partial U) < \varepsilon\} \\
\in [0, 1] & \text{ otherwise}.
\end{cases}
\]
We then calculate, from \eqref{Equ:definition of alpha_s(v)^+} and \eqref{Equ:f_s},
$$
\aligned
0 & \le \int_{U} f_s^- (a_s - b_s) dx \\
& =\int_{U}(|\nabla b_{s}|^{n-2}\nabla b_{s} - |\nabla a_{s}|^{n-2}\nabla a_{s})\cdot 
(\nabla a_{s} - \nabla b_{s})\phi_\varepsilon  \\
&\quad\quad\quad + \int_{U}(|\nabla b_{s}|^{n-2}\nabla b_{s} - |\nabla a_{s}|^{n-2}\nabla a_{s})\cdot 
(a_s - b_s)\nabla\phi_\varepsilon.
\endaligned
$$
That is
\begin{equation}\label{Equ:small}
 0  \le-2^{1-n}\int_{U} |\nabla b_{s} - \nabla a_{s}|^{n}\phi_\varepsilon 
 + C\int_{U}
 (|\nabla b_{s}|^{n-1}+|\nabla a_{s}|^{n-1})(a_{s} - b_{s}) |\nabla\phi_{\varepsilon}|, 
\end{equation}
where we use the following inequality
$$
(|\vec{a}|^{n-2}\vec{a} - |\vec{b}|^{n-2}\vec{b})\cdot(\vec{a}-\vec{b}) \geq 2^{1-n} |\vec{a}-\vec{b}|^n
$$
for any vectors $\vec{a}, \vec{b}\in\mathbb{R}^n$. Since $a_{s}=b_{s}$ on $\partial U$, 
we may assume $|a_{s} - b_{s}|\le C\varepsilon$ in $supp(\nabla\phi_{\varepsilon})$ by their regularities (as we mentioned $\bar U$ is in the interior 
of $D$). We also naturally have $|\nabla\phi_{\varepsilon}|\le\frac{C}{\varepsilon}$ for the cutoff function $\phi_{\varepsilon}$. 
Hence 
$$
(a_s - b_s)|\nabla\phi_\varepsilon| \leq C.
$$
Let $\varepsilon\to0$ and derive a contradiction in \eqref{Equ:small} (we may need to adjust $b_s = \mathcal{A}_s(v) +\lambda$ for $\lambda\in (0, 1]$ so
that $\partial U$ is at least $C^1$ if necessary). Unless
\[
\nabla b_{s}  \equiv \nabla a_{s}  \,\,{\rm in}\,\,U
\]
which is again a contradiction to the assumption of the proof $b_{s}(x_{0}) < a_{s}(x_{0})$.
\endproof

Consequently, we know that
\begin{equation}\label{Equ:v^+ upper b}
v \leq \mathcal{A}(v) + 1.
\end{equation}

\vskip 0.1in
\noindent{\bf Step 4.} In this step we want to show that the point mass 
$(-\Delta_n v) (\{p\})$ has a uniform lower bound for any $p$ in the support $\mathcal{S}$
of the Radon measure $-\Delta_n v$, which completes the proof of Theorem \ref{Thm:theorem-B}. 
By Step 2, we know that $(-\Delta_n v) (\{p\}) = (-\Delta_n v)^+ (\{p\})$ for $p\in \mathcal{S}$. Our argument in this step therefore 
relies on the following strengthened two aspects of \cite[Theorem 1.1]{MQ18}. The first is

\begin{theorem}\label{Thm:improvement of theorem 1.1} 
Let $w$ be a nonnegative lower semi-continuous function that is $n$-superharmonic in $B(0, 2)\subset \r^n$ and 
\[
-\Delta_n u=\mu\ge 0
\]
for a Radon measure $\mu\ge 0$ and the origin is in the support $\mathcal{S}$ of the Radon measure $\mu$. 
Then there is a set $E\subset\mathbb{R}^n$, which is $n$-thin at the origin, such that
\[
\lim_{x\notin E \text{ and } |x|\rightarrow 0}\frac{u(x)}{\log\frac{1}{|x|}}=\liminf_{|x|\rightarrow0}\frac{u(x)}{\log\frac{1}{|x|}}=m\ge 0
\]
and
\[
u(x)\geq m\log\frac{1}{|x|} - C \ \text{ for $x\in B(0, 1)\setminus \mathcal{S}$ and some $C$}.
\]
Moreover, if $(B(0, 2)\setminus \mathcal {S}, e^{2u}|dx|^2)$ is complete at the origin, then $m\geq1$. 
\end{theorem}

The difference of Theorem \ref{Thm:improvement of theorem 1.1} from \cite[Theorem 1.1]{MQ18} is that, when using 
completeness at the origin to derive $m\geq 1$, the assumption of the origin
being isolated singularity is dropped. Theorem \ref{Thm:improvement of theorem 1.1} indeed applies to $\mathcal{A}(v)$, since the completeness at 
the origin for the metric $e^{2v}|dx|^2$ implies the completeness at the origin for the metric $e^{2\mathcal{A}(v)}|dx|^2$ by \eqref{Equ:v^+ upper b}.

\begin{proof}[Proof of Theorem \ref{Thm:improvement of theorem 1.1}] We will only need to focus on the last statement 
to conclude $m\geq 1$. All the rest goes exactly as in \cite{MQ18}. It is easily seen that, if we can find a ray connecting the origin to some point 
in $B(0, 2)\setminus\mathcal{S}$, then the length of such ray under the metric $e^{2u}|dx|^2$ will be finite if $m < 1$. In the following we present a
proof of the existence of such ray in the way similar to the idea used in the proof of \cite[Theorem 5.2]{AH73} in 2 dimensions.
\\

Let $P$ be the radial project
$$
P = \frac {x}{|x|}: B(0, r)\setminus\{0\} \to \mathbb{S}^{n-1}\subset \mathbb{R}^n.
$$ 
Clearly, $P$ is a Lipschitz map. Recall
$$
\omega_i = \{x\in \mathbb{R}^n: 2^{-i-1} \leq |x|\leq 2^{-i}\} \text{ and } \Omega_i = \{x\in \mathbb{R}^n: 2^{-i-2} \leq |x|\leq 2^{-i +1}\}.
$$
Applying \cite[Theorem 5.2.1]{AH99}, we have
$$
cap_n(P(\omega_i\cap E), \Omega_{-1}) \leq C cap_n(\omega_i\cap E, \Omega_i)
$$
and
$$
\aligned
cap_n(\bigcup_{i=0}^\infty P(\omega_i\cap E), \Omega_{-1}) & \leq \sum_{i=1}^\infty
cap_n(P(\omega_i\cap E), \Omega_{-1}) \\
&  \leq C \sum_{i=1}^\infty cap_n(\omega_i\cap E, \Omega_i)\\
& \leq C \sum_{i=1}^\infty 2^{i-1} cap_n(\omega_i\cap E, \Omega_i),
\endaligned
$$
which is finite since $E$ is n-thin at the origin under \cite[Definition 3.1]{MQ18}. Therefore
$$
cap_n(\bigcup_{i=i_0}^\infty P(\omega_i\cap E), \Omega_{-1}) 
$$
is arbitrarily small as $i_0$ chosen to be large. From $cap_n(\mathbb{S}^{n-1}(1), \Omega_{-1}) > 0$ (cf. Lemma \ref{Lem:n-cap-Hausdorff}), 
it then implies that there is a ray from $\omega_{i_0}\setminus E$ to the origin in $B(0, 2^{-i_0})\setminus E$, which completes the proof of 
Theorem \ref{Thm:improvement of theorem 1.1}.
\end{proof}  

Notice that here $\mathcal{S} \subset E$ in Theorem \ref{Thm:improvement of theorem 1.1}. 
The second strengthened aspect of the \cite[Theorem 1.1]{MQ18}  is the following

\begin{lemma}\label{Lem:m=point mass}
Under the assumptions of Theorem \ref{Thm:theorem-B}, when applying Theorem \ref{Thm:improvement of theorem 1.1},
$$
n \omega_n m^{n-1} \leq  \mu(\{0\}).
$$
\end{lemma}

\proof We continue using the notations in the proof of Theorem \ref{Thm:improvement of theorem 1.1}. For $\epsilon > 0$, let
$$
G_{m-\epsilon, L} = (m-\epsilon)\log \frac 1{|x|} + L.
$$
We want to pick up sequence of domains $\{U_k\}$ for a sequence of numbers $L_k\to\infty$ such that
\begin{itemize}
\item $\mathcal{A}(v) \geq G_{\epsilon, L_k}$ in domain $U_k$;
\item $\mathcal{A}(v) = G_{\epsilon, L_k}$ on $\partial U_k$;
\item $\bigcap_{k=1}^\infty U_k = \{0\}$.
\end{itemize}
Notice that this requires that $\partial U_k\cap E = \emptyset$. Assume such $\{U_k\}$ exists. Then one observes that, on the boundary
$\partial U_k$,  
\begin{equation}\label{Equ:nabla v^+ greater nabla G}
\aligned
\nabla \mathcal{A}(v) - \nabla G_{\epsilon, k} & = (\frac{\partial (\mathcal{A}(v) - G_{\epsilon, k})}{\partial n}, 0, \cdots, 0)\\
\frac{\partial \mathcal{A}(v) }{\partial n} & \geq \frac {\partial G_{\epsilon, k}}{\partial n}.
\endaligned
\end{equation}
Therefore, by the divergence theorem,
\begin{equation}\label{Euq:divergence theorem}
\aligned
(-\Delta_n \mathcal{A}(v)) (U_k) & = \int_{\partial U_k} |\nabla \mathcal{A}(v)|^{n-2}\frac {\partial \mathcal{A}(v)}{\partial n} \\
& \geq \int_{\partial U_k} |\nabla G_{\epsilon, k}|^{n-2}\frac {\partial G_{\epsilon, k}}{\partial n} \\
& = n\omega_n (m-\epsilon)^{n-1}
\endaligned
\end{equation}
which implies this lemma. So the rest is to construct the sequence $\{U_k\}$. 
\\

In the light of \cite[Lemma 5.1]{MQ18} (also see \cite[Lemma 1.4 page 212]{Re94} and \cite[Theorem 4]{Ge61}), 
for a straight line segment $l_i$ connecting the inner sphere to the outer sphere of $\omega_i$, we have
$$
cap_n (l_i, \Omega_i) = cap_n(l_0, \Omega_0) \geq  \frac {c_n}{(\log 3)^{n-1}} > 0.
$$ 
Hence there are plenty shperes $\mathbb{S}^{n-1}(\lambda) \subset \omega_i\setminus E$ at least for $i$ sufficiently large. We may consider
a sequence $r_k\to 0$ such that the spheres $\mathbb{S}^{n-1}(r_k) \bigcap E = \emptyset$. Let
$$
L_k = \min_{\mathbb{S}^{n-1}(r_k)} (\mathcal{A}(v) - (m-\epsilon)\log\frac 1{|x|}) - 1.
$$
Then
$$
\mathcal{A}(v) \geq G_{\epsilon, k} + 1 > G_{\epsilon, k} \text{ on $\mathbb{S}^{n-1}(r_k)$}.
$$
Because
\begin{itemize}
\item $G_{\epsilon, k} (x)  \to -\infty$ as $x\to \infty$;
\item $\mathcal{A}(v) (x) \to + \infty$ as $x$ approaches $\mathcal{S}$ away from the origin;
\item $\mathcal{A}(v) (x) \geq m\log \frac 1{|x|} -C$ as $x\to 0$ away from $\mathcal{S}$. 
\end{itemize}
There exists $U_k$ as desired. Finally, notice that
$\mathcal{A}(v)$ is at least $C^{1, \gamma}$ away from the singular set $\mathcal{S}$ due to \cite[Theorem 3.3]{Te14}. The regularity of $U_k$ therefore
is guaranteed. This completes the proof of this lemma.
\endproof

Finally Lemma \ref{Lem:m=point mass} and Theorem \ref{Thm:improvement of theorem 1.1} together produce a contradiction if the 
singular set $\mathcal{S}$ is infinite. Now we complete the proof of Theorem \ref{Thm:theorem-B}.
\endproof

To end this subsection we state a corollary to Theorem \ref{Thm:theorem-B} in a more geometric fashion. 

\begin{corollary}\label{Cor:cheng-yau}
For $n\geq 3$, let $\Omega$ be a domain in the standard unit round sphere $(\mathbb{S}^n, g_{\mathbb{S}})$.
Suppose that, on $\Omega$, there is a complete conformal metric $g= e^{2u}g_{\mathbb{S}}$ satisfying either $Ric_g$ is nonnegative outside
a compact subset or
\begin{enumerate}
\item $Ric^{-}_g \in L^{1}(\Omega, g) \cap L^\infty(\Omega, g)$
\item $R_g \in L^{\infty}(\Omega, g) \text{ and } |\nabla^g R_g|\in L^{\infty}(\Omega, g)$. 
\end{enumerate}
Then $\partial\Omega = \mathbb{S}^n\setminus\Omega$ is a finite point set. 
\end{corollary}

\begin{proof} We only need to verify the two assumptions of Theorem \ref{Thm:theorem-B}. First, since $g$ is complete and the
scalar curvature is bounded, directly applying \cite[Proposition 8.1]{CHY04}, we know 
\[
\lim_{x\to\partial\Omega}w(x)=+\infty.
\]
Now,  via a stereographic projection, let $e^{2w}g_{\mathbb{S}}=e^{2u}|dx|^2.$
We want to prove
\[
Ric_g^{-}|\nabla u|^{n-2} e^{2u} \in L^{1}(\Omega, |dx|^2)
\]
instead, which is equivalent to one of the second assumptions in Theorem \ref{Thm:theorem-B}. From the scalar curvature equation
\[
0=-\frac{4(n-1)}{(n-2)}\Delta_{g}(e^{-\frac{n-2}{2}u})+R_{g}e^{-\frac{n-2}{2}u},
\]
and Cheng-Yau's gradient estimate (cf. \cite[Theorem 2.12, Chapter VI]{SY94} and \cite[Lemma 2.3]{CQY00}), under the assumptions
\begin{enumerate}
\item $Ric_g $ is bounded from below;
\item both $R_g \in L^{\infty}(\Omega, g) \text{ and } |\nabla^g R_g|\in L^{\infty}(\Omega, g)$, 
\end{enumerate}
we know that 
\[
|\nabla_{g} u|\le C \text{ and equivalently }
|\nabla u|\le e^{u},
\]
which implies 
\[
\int_M Ric_g^{-}|\nabla u|^{n-2}e^{2u}\le\int_M Ric_g^{-}e^{nu}\le\|Ric^{-}\|_{L^{1}(\Omega,g)}.
\]
So Theorem \ref{Thm:theorem-B} applies.
\end{proof}

Before we end this subsection we want to share examples to 
compare Theorem \ref{Thm:theorem-B} and Corollary \ref{Cor:cheng-yau} to \cite[Theorem 2.1]{CH02} and its improvement \cite[Corollary 2.4]{Car20}
in cases when ambient manifolds are round spheres. 
Particularly, no condition on the upper bound of Ricci is assumed in Theorem \ref{Thm:theorem-B} and Corollary \ref{Cor:cheng-yau}.
Let us consider 
a conformal metric $g = e^{2u}|dx|^2 = v^\frac 4{n-2}|dx|^2$ on the entire Euclidean space $\mathbb{R}^n = \mathbb{S}^n\setminus\{p_0\}$, where
$$
v(x) = v(|x|) = (\frac 2{1+ |x|^2})^{\frac {m(n-2)}4} \text{ and } u =   \frac m2 \log \frac 2{1+ |x|^2} \text{ when $|x|$ is large}.
$$
We may deform $g$ to have negative Ricci in the region when $|x|$ is not large. For completeness and noncompactness, $m\in [0, 1]$. 
Hence its scalar curvature is
$$
\aligned
R_g & =  \frac {4(n-1)}{n-2} (-\Delta v)v^{-\frac {n+2}{n-2}} \\
& =  n(n-1) (\frac m2 + \frac {m(2-m)(n-2)}{4n} |x|^2) (\frac 2{1+|x|^2})^{(2-m)}\\
\endaligned
$$
and Ricci curvature tensor is
$$
\aligned
 (Ric_g)_{ij}  & = -(n-2)u_{ij} +(n-2)u_iu_j - \Delta u\delta_{ij} -(n-2)|\nabla u|^2\delta_{ij}\\
&  = \frac m 2 (n-1) (\frac 2{1+|x|^2})^2 \delta_{ij} + \frac {m(2-m) (n-2)|x|^2}4 (\frac 2{1+|x|^2})^2 \delta_{ij} \\
& \quad\quad - \frac {m(2-m)(n-2)x_ix_j }4  (\frac 2{1+|x|^2})^2.
\endaligned
$$
For any $m\in (0, 1)$, we have
$$
\aligned
Vol^g(B_r(0)) & \lesssim \int_0^r |x|^{-mn}|x|^{n-1}d|x| \lesssim r^{n-mn} \\
\text{ and } Vol^g(B_r(x_0)) & \lesssim Vol^{|dx|^2} (B_r(x_0)) \lesssim r^n \text{ for any $x_0\in \mathbb{R}^n$}
\endaligned
$$
$$
\int_{\mathbb{R}^n} |R^+_g|^\frac n2 dvol_g \simeq \int_0^\infty |x|^{-n - mn} |x|^{mn}|x|^{n-1} \simeq \infty
$$
and similarly
$$
\int_{\mathbb{R}^n} |Ric_g|^\frac n2 dvol_g \simeq \int_0^\infty |x|^{-n-mn} |x|^{mn} |x|^{n-1} \simeq\infty.
$$
On the other hand, $Ric_g$ clearly turns nonnegative when $|x|$ is sufficiently large and is not asymptotically vanishing. Therefore 
the assumption $Ric \in L^\frac n2(M, g)$ in \cite[Theorem 2.1]{CH02} and \cite[Corollary 2.4]{Car20} is not satisfied. But Theorem \ref{Thm:theorem-B} 
and Corollary \ref{Cor:cheng-yau} cover the cases in these examples.


\section{Injectivity of conformal immersions}\label{Sec:injectivity}

In \cite[Theorem 3]{CQY00}, in order to identify the compactification, it invoked the injectivity theorem of Schoen-Yau \cite{SY88} (cf. 
\cite{CH06, LUY20, CL20}). The injectivity theorem
of Schoen-Yau is available for \cite[Theorem 3]{CQY00} because the scalar curvature is assumed to be strictly positive. In this section we 
consider manifolds that have a conformal immersion 
$$
\Phi: (M^n, g)\to (\mathbb{S}^n, g_{\mathbb{S}})
$$
instead simply are domains in $\mathbb{S}^n$. Our approach follows closely to what was in \cite{SY88} but we would like to have the injectivity when the 
scalar curvature is allowed to have some negativity. We believe the discussion in this section has importance in its own, besides it strengthens 
Theorem \ref{Thm:theorem-B}.
\\

To describe the end behavior that is related to the injectivity, Schoen-Yau in \cite{SY88} introduced the following very important quantity: 
\begin{equation}\label{Equ:d(m)}
d(M) =  \frac {n-2}2 \inf \{\gamma: \text{ such that } \int_{M^\setminus \mathcal{O}} G_o^\gamma dvol_g < \infty\}
\end{equation}
where $G_o$ is the positive minimal Green function for the conformal Laplacian at a base point $o\in M^n$ and $\mathcal{O}$ is an open neighborhood of 
$o$ in $M^n$. First, Schoen-Yau made the following estimate of $d(M)$ when the scalar curvature is nonnegative.

\begin{proposition}\label{Prop:schoen-yau} (Schoen-Yau \cite[Proposition 2.4]{SY88})
Suppose that $(M^n, g)$ is complete and that there is a conformal immersion $\Phi: (M^n, g) \to (\mathbb{S}^n, g_{\mathbb{S}})$. Assume the scalar
curvature is nonnegative. Then $d(M) \leq \frac n2$.
\end{proposition}

It is remarkable that Schoen-Yau managed to use $d(M)$ to gain the injectivity.

\begin{theorem}\label{Thm:schoen-yau} (Schoen-Yau \cite[Theorem 3.1]{SY88})
Suppose that $(M^n, g)$ is complete and that there is a conformal immersion $\Phi: (M^n, g) \to (\mathbb{S}^n, g_{\mathbb{S}})$.
Assume the scalar curvature of $(M^, g)$ is bounded from below and $d(M) < \frac{(n-2)^2}n$. For the dimension $n = 3,4$, assume in addition,
the scalar curvature is bounded. Then $\Phi$ is an embedding.
\end{theorem}

As a consequence of Proposition 2.4 and Theorem 3.1 in \cite{SY88}, they derived the injectivity when the scalar curvature is nonnegative and the 
dimension is greater than 6. The proof of the rest of the injectivity theorem of Schoen-Yau \cite{SY88} used a different argument based on a variant
of positive mass theorem (cf. \cite{CH06, LUY20, CL20}). 
\\

We will produce a version of Proposition \ref{Prop:schoen-yau} and Theorem \ref{Thm:schoen-yau} 
under different assumptions. We will then prove the main theorem as an improvement of Theorem \ref{Thm:theorem-B}.


\subsection{Estimates of $d(M)$}\label{Subsec:estimate of d(M)} 
In this subsection we want to state and prove a variant of Proposition 
\ref{Prop:schoen-yau} to give an estimate of $d(M)$ under integral lower bound of the scalar curvature. The proof is a modificationof the proof
of Proposition 2.4 in \cite{SY88}.  

\begin{proposition}\label{Prop:d(M)-estimate}
Suppose that $(M^n, g)$ ($n\geq 3$) is a manifold that has a conformal immersion 
$$
\Phi: (M^n, g) \to (\mathbb{S}^n, g_{\mathbb{S}}).
$$ 
Assume that $R_g^-\in L^\frac n2(M, g)$. Then $d(M) \leq \frac n2$.
\end{proposition}

\begin{proof}

Given $o\in M$, we let $G_o$ be the minimal Green function with pole
at $o$. Let $\mathcal{O}$ be an open neighborhood of $o$. Let $\{U_{i}\}_{i=1}^{\infty}$
be an exhaustion of $M$ such that $\bar{U}_{i}\subset U_{i+1}$ and $\bar{U}_i$ is compact 
for each $i=1, 2, \cdots$. Consider the positive Green function $G_{i}$ that satisfies
$$
\left\{\aligned
- \Delta G_{i} + \frac {n-2}{4(n-1)}R_g G_{i} & =\delta_{o}.\\
G_{i}|_{\partial U_{i}} & =0
\endaligned\right.
$$
Then $G_{i} \to G_o$ on $M$ in $C^\infty$ over any compact subset since $G_o$ is the minimal Green function.  
On $U_{i}\backslash\mathcal{O}$, we have 
\[
-\Delta G_{i}+ \frac {n-2}{4(n-1)} R^{+}_g G_{i} = \frac {n-2}{4(n-1)} R^{-}_g G_{i},
\]
where $R_g^{+}$ and $R_g^{-}$ are the positive and negative part of the scalar curvature $R_g$ respectively. For simplicity, we will
use $c(n) = \frac {n-2}{4(n-1)}$ in the following. Multiplying by $G_{i}^{\varepsilon}$ for any small $\varepsilon>0$ and integrating on $U_{i}\backslash\mathcal{O}$ on both sides,  
we have 
\[
\int_{U_{i}\backslash\mathcal{O}}-G_{i}^{\varepsilon}\Delta G_{i} + c(n) R_g^{+}G_{i}^{1+\varepsilon}
= c(n) \int_{U_{i}\backslash\mathcal{O}} R_g^{-}G_{i}^{1+\varepsilon}.
\]
Then we have 
\begin{equation}\label{Equ:difference-0}
 \int_{U_{i}\backslash\mathcal{O}}  \frac {4\varepsilon}{(1+\varepsilon)^2}|\nabla G_{i}^{\frac{1+\varepsilon}{2}}|^{2} 
+ c(n)R_g^{+}G_{i}^{1+\varepsilon} \le c(n)\|R_g^{-})\|_{L^{\frac{n}{2}}(M\setminus\mathcal{O})} \, (\int_{U_{i}\backslash\mathcal{O}}
 G_{i}^{\frac{n(1+\varepsilon)}{n-2}})^{\frac{n-2}{n}} + C,
\end{equation}
where we use $\int_{\partial\mathcal{O}}\frac{\partial G_{i}}{\partial\nu}G_{i}^{\varepsilon}\to\int_{\partial\mathcal{O}}\frac{\partial G}
{\partial\nu}G^{\varepsilon}$. Notice that, different from \cite{SY88}, we use the integral lower bound of the scalar curvature instead. 
\\

Suppose $\mathcal{O}\subset U_{1}$, let $\zeta$ be a cut off function
which is $1$ in $\mathcal{O}$ and $0$ outside $U_{1}$. By the Sobolev inequality from \cite[Proposition 2.2]{SY88},
\begin{align*}
(\int_{U_{i}\backslash\mathcal{O}}G_{i}^{\frac{n(1+\varepsilon)}{n-2}})^{\frac{n-2}{n}} 
 & \le C \int_{U_{i}}[(1-\zeta)G_{i}]^{\frac{n(1+\varepsilon)}{n-2}})^{\frac{n-2}{n}} + C\\
 & \le C (\int_{U_{i}}|\nabla((1-\zeta)G_{i})^{\frac{1+\varepsilon}{2}}|^{2}+c(n)R_g^{+}((1-\zeta)G_{i})^{1+\varepsilon})+C\\
 & \le C (\int_{U_{i}}|\nabla (G_{i})^{\frac{1+\varepsilon}{2}}|^{2}+c(n) R_g^{+}G_{i}^{1+\varepsilon})+C
\end{align*}
Therefore we arrive at 
\begin{align*}
 &  \int_{U_{i}\backslash\mathcal{O}}\frac {4\varepsilon}{(1+\varepsilon)^2}|\nabla G_{i}^{\frac{1+\varepsilon}{2}}|^{2}+c(n)R_g^{+}G_{i}^{1+\varepsilon} \\
\le & C \|R_g^{-}\|_{L^{\frac{n}{2}} (M\setminus\mathcal{O})}
[\int_{U_{i}\backslash\mathcal{O}}\frac {4\varepsilon}{(1+\varepsilon)^2}|\nabla G_{i}^{\frac{1+\varepsilon}{2}}|^{2} 
+ c(n)R_g^{+}G_{i}^{1+\varepsilon}] + C 
\end{align*}
Thus, by the assumption $R_g^-\in L^\frac n2(M, g)$,  we choose $\mathcal{O}$ such that $\|R^{-}_g\|_{L^{\frac{n}{2}}(M\backslash\mathcal{O})}$
is sufficiently small and get
\[
\int_{U_{i} \setminus\mathcal{O}}|\nabla G_{i}^{\frac{1+\varepsilon}{2}}|^{2}+c(n) R_g^{+}G_{i}^{1+\varepsilon}\le 
C(\varepsilon,\mathcal{O},\|R^{-}_g\|_{L^{\frac{n}{2}}(M\backslash\mathcal{O})}),
\]
and
\[
\int_{U_{i} \setminus\mathcal{O}}G_{i}^{\frac{n(1+\varepsilon)}{n-2}}\le C(\varepsilon).
\]
This finishes the proof of this proposition.
\end{proof}


\subsection{Injectivity from the integral lower bound on Ricci curvature}\label{Subsec:d(M)-injectivity} 
Let us first state the injectivity theorem in term of integral lower bound on Ricci curvature. Recall from Lemma \ref{Lem:ricci-parabolicity} in Section \ref{Subsec:capacity-parabolicity}, the integral lower bound on Ricci curvature implies p-parabolicity.

\begin{theorem}\label{Thm:ricci-injectivity}
Suppose that $(M^n, g)$ is a complete, noncompact manifold. And suppose that there is a conformal immersion
$$
\Phi: (M^n, g) \to (\mathbb{S}^n, g_{\mathbb{S}}).
$$
Assume that $Ric^-_g \in L^\frac p2(M, g)$ for some $p\in (n, \infty]$. Then $\Phi$ is injective provided that
$$
d(M)< \frac {(n-2)^2p}{(p-2)n}.
$$

\end{theorem}

\proof Our proof is a modification of the proof of \cite[Theorem 3.1]{SY88}. Readers are referred to more details in \cite{SY88}.
As in \cite{SY88}, for a point $o\in M$, let $G_o$ be the minimal Green function  and $\bar G_o$ be the pul-back of the Green function 
on the sphere. And let
$$
v = \frac {G_o}{\bar G_o}.
$$
Then, as observed in \cite{SY88}, 
$$
\tilde g = G_o^\frac 4{n-2} g = v^\frac 4{n-2} \bar g \text{ and } \bar g = \bar G_o^\frac 4{n-2} g
$$
where $\bar g$ is the pull-back from the Euclidean metric on $\mathbb{S}^n\setminus\{\Phi(o)\}$. Hence $v$ is a harmonic function with respect to the flat
metric $\bar g$ because both $\tilde g$ and $\bar g$ are scalar-flat. By the definition, it is easily seen that $0<v\leq 1$. The goal is to show that $v\equiv 1$. Starting from the Bochner formula \cite[(3.1)]{SY88} for $v$
$$
\bar \Delta |\bar\nabla v|^2 = 2 |\bar\nabla\bar\nabla v|^2, 
$$
for $\alpha = \frac {2(n-2)}n$, it is shown that  
$$
\int_{M}\phi^{2}|\bar{\nabla}v|^{\alpha-2}|\bar{\nabla}|\bar{\nabla}v||^{2}d\bar{vol} \le C\int_{M}|\bar{\nabla}\phi|^{2}|\bar{\nabla}v|^{\alpha}d\bar{vol}
=  C\int_{M}|\nabla\phi|^{2}\bar{G}_{o}^{\alpha}|\nabla v|^{\alpha}dvol
$$
(cf. \cite[(3.4)]{SY88}), where quantities with an upper bar are those taken with respect to the metric $\bar g$, 
for any smooth function $\phi$ with compact support. Hence
\begin{equation}\label{Equ:difference-1}
\int_{M}\phi^{2}|\bar{\nabla}v|^{\alpha-2}|\bar{\nabla}|\bar{\nabla}v||^{2}d\bar{vol} \le 
C(\int_{M}|\nabla\phi|^{p} dvol)^{\frac{2}{p}}(\int_{supp\nabla\phi}(\bar{G}_{o}^{\alpha}|\nabla v|^{\alpha})^{\frac{p}{p-2}}dvol)^{1-\frac{2}{p}}
\end{equation}
for any smooth function with compact support on $M$. \eqref{Equ:difference-1} is where we deviate from \cite{SY88}. We will handle the first factor
of the right hand side of \eqref{Equ:difference-1} by p-parabolicity, which is available thanks to Lemma \ref{Lem:ricci-parabolicity} in Section 
\ref{Subsec:capacity-parabolicity}. According to
Definition \ref{Def:p-capacity} and Definition \ref{Def:p-parabolicity} in Section \ref{Subsec:capacity-parabolicity}, there is a smooth function 
$\phi$ of compact support that is equal to $1$ inside a big ball $B(o, 2r)$ and equal to $0$ outside an even bigger ball $B(o, R)$ such that 
$\int_M |\nabla \phi|^p dvol$ is as small as desired for appropriate choice of $r$ and $R$. 
\\

To estimate the second factor in the right hand side of \eqref{Equ:difference-1} is to estimate
\begin{equation}\label{Equ:second factor} 
\int_{supp\nabla\phi} \bar G_o^\beta |\nabla v|^\beta \text{ where }
\beta =  \frac {2p(n-2)}{n(p-2)} \in (0, 2)
\end{equation}
since $p\in (n, \infty]$. Here the same idea of the argument in \cite{SY88} applies. For the convenience of readers, we include an outline. Notice that
$$
\bar G_o^\beta |\nabla v|^\beta \leq C (|\nabla G_o|^\beta + G_o^\beta |\nabla \log \bar G_o|^\beta)
$$
due to the definition of $v$. Let $\gamma = \frac 12 \beta(2-\beta)$. Then 
$$
|\nabla G_o|^\beta = G_o^\gamma G_o^{-\gamma} |\nabla G_o|^\beta \leq \frac {2-\beta}2 G_o^\beta + \frac \beta 2
G_o^{\beta} |\nabla \log G_o|^2
$$
and
$$
G_o^\beta |\nabla \log \bar G_o|^\beta \leq G_o^\beta (\frac {2-\beta}2 + \frac \beta 2|\nabla\log\bar G_o|^2)\leq
\frac {2-\beta}2 G_o^\beta + \frac \beta 2 G_o^{\beta} |\nabla \log \bar G_o|^2
$$
by Young's inequality. To treat $G_o^\beta |\nabla\log\bar G_o|^2$, we realize, away from the poles,
$$
\Delta\log\bar G_o - \Delta \log G_o = |\nabla\log G_o|^2 - |\nabla\log\bar G_o|^2,
$$
which helps to derive
$$
\int_M \psi^2 G_o^\beta |\nabla\log\bar G_o|^2 \leq C \int_M \psi^2 G_o^\beta |\nabla\log G_o|^2 + C\int_M \psi^2 G_o^\beta
$$
and \cite[(3.6)]{SY88}
\begin{equation}\label{Equ:3.6}
\int_M \psi^2 \bar G_o^\beta |\nabla v|^\beta \leq C \int_M \psi^2 G_o^{\beta-2} |\nabla G_o|^2 + C\int_M \psi^2 G_o^\beta
\end{equation}
for any smooth function $\psi$ with a compact support away from the point $o$. In fact we will choose $\psi$ that is
identically one on $\bar B(o, R)\setminus B(0, 2r)$ and vanishes outside $\bar B(o, 2R)\setminus B(0, r)$
which includes the support of $\phi$. Finally, to estimate the first term in the right hand side 
of \eqref{Equ:3.6}, we multiply $G^{\beta -1}\psi^2$ to both sides of the equation
$$
- \Delta G_o = - \frac{n-2}{4(n-1)}R_g G_o.
$$
and integrate to get
$$
\aligned
\int_M \psi^2 G_o^{\beta -2}|\nabla G_o|^2 & \leq C \int_M\psi^2R_g^-G_o^\beta  + C\int_M G_o^\beta |\nabla\psi|^2 \\
& \leq C\|R_g^-\|_{L^\frac p2(M, g)} (\int_M (\psi^2 G_o^\beta)^\frac p{p-2})^\frac {p-2}p + C\int_M G_o^\beta |\nabla\psi|^2
\endaligned
$$
Similar to \cite{SY88}, when 
$$
d(M) < \frac {p(n-2)^2}{n(p-2)} = \frac {n-2}2 \beta,
$$
we have
$$
\int_{M\setminus\mathcal{O}} G_o^\beta <\infty \text{ and } \int_{M\setminus\mathcal{O}} G_o^{\beta\frac p{p-2}} <\infty.
$$
Thus we may arrange the sizes of $r$ and $R$ to derive
$$
\int_{B(o, 2r)}|\bar{\nabla}v|^{\alpha-2}|\bar{\nabla}|\bar{\nabla}v||^{2}d\bar{vol} 
$$
is as small as desired and therefore vanishes, which implies $v\equiv 1$ as argued in \cite{SY88}. 
\endproof

Combining with Proposition \ref{Prop:d(M)-estimate} from the previous subsection we have

\begin{corollary}\label{Cor:ricci-injectivity}
Suppose that $(M^n, g)$ is a complete, noncompact manifold. And suppose that there is a conformal immersion
$$
\Phi: (M^n, g) \to (\mathbb{S}^n, g_{\mathbb{S}}).
$$
Assume that $Ric^-_g \in L^\frac p2(M, g)\cap L^\frac n2(M, g)$ for some $p\in (n, \infty]$. Then $\Phi$ is injective provided that
$n\geq 5$.
\end{corollary}

This is because of the interpolation property of $L^p$ norms and 
$$
\lim_{p\to n} \frac {p(n-2)^2}{n(p-2)} = n-2 > \frac n2 \text{ when $n\geq 5$}.
$$


\subsection{Injectivity from n-parabolicity}\label{Subsec:n-parabolicity-injectivity}
In this subsection we prove an injectivity theorem using n-parabolicity only. Namely,

\begin{theorem}\label{Thm:n-parabolicity-injectivity}
Suppose that $(M^n, g)$ ($n\geq 3$) is a complete noncompact manifold that has a conformal immersion
$$
\Phi: (M^n, g)\to(\mathbb{S}^n, g_{\mathbb{S}}).
$$
Assume that $(M^n, g)$ is n-parabolic. Then $\Phi$ is injective.
\end{theorem}

\proof We continue to use the same set up and notations in the previous subsection. The goal is again to show that $v\equiv 1$ under the 
different assumption. The proof goes exactly the same as the proof of Theorem \ref{Thm:theorem-B}  in the previous subsection (cf. cite{SY88})
until \eqref{Equ:difference-1}, where $p=n$ now, that is,
\begin{equation}\label{Equ:difference-2}
\int_{M}\phi^{2}|\bar{\nabla}v|^{\alpha-2}|\bar{\nabla}|\bar{\nabla}v||^{2}d\bar{vol} \le 
C(\int_{M}|\nabla\phi|^{n} dvol)^{\frac{2}{n}}(\int_{supp\nabla\phi}\bar{G}_{o}^2|\nabla v|^2 dvol)^{1-\frac{2}{n}}
\end{equation}
The first factor in the right hand side of \eqref{Equ:difference-2} can be handled in the same spirit as before. It is as small as desired by n-parabolicity
assumption for appropriate choices of $\phi$. 
\\

The second factor in the right hand side of  \eqref{Equ:difference-2} becomes
$$
\int_{supp\nabla\phi} \bar G_o^2 |\nabla v|^2 dvol = \int_{supp\nabla\phi} |\bar \nabla v|^2 \bar {dvol}.
$$
In the following we want to show it is bounded for appropriate choices of $\phi$. Again using the fact $v$ is harmonic with respect to the flat metric 
$\bar g$, for a smooth function $\psi$ that has a compact support away from the other poles of $\bar G_o$ and is identically equal to one in a 
neighborhood $B(o, 2\rho)$ of $o$ for some $\rho > 0$ (chosen later), we have
$$
\aligned
0 & = \int_{M\setminus B(o, \rho)}   \psi^2 v (-\bar\Delta v) \bar {dvol}
 = \int_{M\setminus B(o, \rho)} \psi^2 |\bar\nabla v|^2 \bar{dvol}  \\ 
&  - 2\int_{M\setminus B(o, \rho)} \psi v \nabla v\cdot\nabla\psi \bar{dvol}  - \int_{\partial B(o, \rho)} v\frac{\partial v}{\partial \bar n} \bar{d\sigma}
\endaligned
$$ 
which implies
$$
\int_{M\setminus B(o, \rho)} \psi^2 |\bar\nabla v|^2 \bar{dvol}  \leq 4 \int_M |\bar\nabla \psi|^2  v^2 \bar{dvol} 
+ 2\int_{\partial B(o, \rho)} v\frac{\partial v}{\partial \bar n} \bar{d\sigma}.
$$
The second term is a boundary integral on $\partial B(o, \rho)$, which is bounded when $\rho$ is fixed.
About $\psi$, applying the same argument as in \cite{SY88}, we may remove the restrictions of vanishing in a neighborhood of the other poles of
$\bar G_o$. We therefore have 
\begin{equation}\label{Equ:extension of psi}
\aligned
\int_M |\bar\nabla \psi|^2 v^2 \bar{dvol} =&  \int_M G_o^2 |\nabla \psi|^2 dvol\\
& \leq  (\int_{supp\nabla\psi} |\nabla\psi|^n dvol)^\frac 2n (\int_{supp\nabla\psi} G_o^\frac{2n}{n-2} dvol)^\frac {n-2}n
\endaligned
\end{equation}
for any smooth functions that has compact support and is identically one in $B(o, 2\rho)$. Now the first factor in the right hand side of
\eqref{Equ:extension of psi} can be handled by
the n-parabolicity assumption and the second factor of the right hand side of \eqref{Equ:extension of psi}  
is bound due to \cite[Corollary 2.3]{SY88}. Thus the second factor in the right hand
side of \eqref{Equ:difference-2} is bounded for appropriate choices of $\psi$ after $\phi$ is chosen such as
\begin{itemize}
\item  $\psi$ is identically one on the support of $\nabla\phi$
\item  $\psi$ is identically one in $B(o, 2\rho)$
\item  $\rho < r$ fixed appropriately.
\end{itemize} 
One may fix $\rho>0$ first, then fix $r > 4\rho$,  and finally set $R$ as large as desired. 
\endproof

n-parabolicity is a very interesting notion in conformal geometry though its significance needs to be explored more. For one thing, n-parabolicity is
conformally invariant. Therefore it seems to capture that the Euclidean end is the same as the punctured point most efficiently. The injectivity theorem in
this subsection is worth to mention even though it detached from rest of the discussions in this paper. It is very desirable to find some geometric
conditions that can induce n-parabolicity.  

\subsection{Improvements of Theorem \ref{Thm:theorem-B}}\label{Subsec:proof of main theorem}

In this subsection, as consequences of results in Section \ref{Subsec:estimate of d(M)} and \ref{Subsec:d(M)-injectivity}, 
we want to state the improvement of Theorem \ref{Thm:theorem-B}, particularly, from 
Corollary \ref{Cor:cheng-yau} and Corollary \ref{Cor:ricci-injectivity}, which is the main theorem of this paper.

\begin{theorem}\label{Thm:theorem-C} Suppose that $(M^n, g)$ ($n\geq 5$) is a complete noncompact manifold. And suppose 
that there is a conformal immersion
$$
\Phi: (M^n, g) \to (\mathbb{S}^n, g_{\mathbb{S}}).
$$ 
Assume that either $Ric_g$ is nonnegative outside a compact subset or
\begin{enumerate}
\item $Ric^{-}_g \in L^{1}(\Omega, g) \cap L^\infty(\Omega, g)$
\item $R_g \in L^{\infty}(\Omega, g) \text{ and } |\nabla^g R_g|\in L^{\infty}(\Omega, g)$. 
\end{enumerate}
Then $\Phi$ is an embedding and $\partial \Phi(M) = \mathbb{S}^n\setminus\Phi(M)$ is a finite point set. 
\end{theorem}


\end{document}